\documentclass[a4paper,reqno]{amsart}%reqno so equation numbers don't look like they could be enumeration numbers

\renewcommand{\keywords}[1]{{\def\and{\unskip, }\textbf{Keywords}\enspace#1}}
\newcommand{\subclass}[1]{{\def\and{\unskip; }\textbf{MSC (2010)}\enspace#1}} 
\newcommand{\gobbletwo}[2]{}
\renewcommand{\email}[1]{#1}

\usepackage{iftex}
\RequirePDFTeX%haven't tested with others - font specification would probably need to be adapted at least

\usepackage[utf8]{inputenc}
\usepackage[T1]{fontenc}
\usepackage{lmodern}
\usepackage{amsmath,amssymb} %amssymb provides \restriction and (indirectly) \mathbb

\let\oldopenbox\openbox
\let\openbox\relax
\usepackage{amsthm}
\let\openbox\oldopenbox

\makeatletter
\renewenvironment{proof}[1][\proofname]{\par
\pushQED{\qed}%
\normalfont \topsep6\p@\@plus6\p@\relax
\trivlist
\item\relax
{\itshape
#1\@addpunct{.}}\hspace\labelsep\ignorespaces
}{%
\popQED\endtrivlist\@endpefalse
}
\makeatother

\let\originalleft\left
\let\originalright\right
\renewcommand{\left}{\mathopen{}\mathclose\bgroup\originalleft}
\renewcommand{\right}{\aftergroup\egroup\originalright}

\usepackage{graphicx}

\usepackage{aliascnt}%autoref and multiple theorem environments with the same counter
\usepackage[normalem]{ulem}%strikeout
\usepackage{xcolor}
\usepackage{xparse}
\usepackage{xstring}
\usepackage{accents}%for undercheck
\usepackage{csquotes}
\usepackage{xr-hyper}
\usepackage[hidelinks]{hyperref}

\newcommand{\R}{\mathbb R}
\newcommand{\loneextR}{[-\infty,\infty]}
\newcommand{\extR}{\loneextR}

\NewDocumentCommand{\Time}{o}{\IfNoValueTF{#1}{\R_+}{\lcro{#1,\infty}}}
\NewDocumentCommand{\Paths}{o}{C(\Time[#1])}
\NewDocumentCommand{\Spacetime}{o}{\Paths[#1] \times \Time[#1]}
\newcommand{\G}{\mathcal G}
\newcommand{\Fnull}{\mathcal{F}^0}
\newcommand{\F}{\mathcal F}
\renewcommand{\P}{\mathbb P}%\mathparagraph or \textparagraph..??
\newcommand{\cost}{c}
\newcommand{\costS}{c'}

\newcommand{\push}[2]{#1_*(#2)}
\newcommand{\proj}[1]{\mathsf{proj}_{#1}}
\newcommand{\indicator}[1]{1_{#1}}
\newcommand{\compl}[1]{#1^c}

\newcommand{\restr}[2]{#1_{\restriction#2}}
\newcommand{\mrestr}[2]{#1_{\restriction#2}}
\newcommand{\invimage}[2]{#1^{-1}\left[#2\right]}
\newcommand{\ftimes}{\times}
\newcommand{\mtimes}{\otimes}
\newcommand{\atimes}{\otimes}
\newcommand{\Borel}[1]{\mathcal{B}\left(#1\right)}
\newcommand{\Lebesgue}{\mathcal{L}}%Lebesgue measure
\newcommand{\overbar}[1]{\mkern 1.5mu\overline{\mkern-1.5mu#1\mkern-1.5mu}\mkern 1.5mu}
\newcommand{\closure}[1]{\overbar{#1}}
\newcommand{\rmax}[2]{#1^*(#2)}
\newcommand{\rmaxP}[2]{#1^*_{#2}}
\newcommand{\abs}[1]{\left|#1\right|}
\newcommand{\Lnorm}[2]{\lVert #2 \rVert_{L^{#1}}}
\renewcommand{\d}[1]{\,d#1}%\diacritic
\newcommand{\E}{\mathbb E}
\newcommand{\cExp}[3][0]{\E\left[#3\middle|\smash{\F^{#1}_{#2}}\right]}
\newcommand{\RST}{\mathsf{RST}}
\newcommand{\RSTt}[1]{\mathsf{RST}^{#1}}
\newcommand{\RSTy}[1]{\mathsf{RST}_{#1}}
\newcommand{\RSTty}[2]{\mathsf{RST}^{#1}_{#2}}
\newcommand{\RSTyd}[2]{\mathsf{RST}_{#1}(#2)}
\newcommand{\RSTtyd}[3]{\mathsf{RST}^{#1}_{#2}(#3)}
\newcommand{\RSTtf}[1]{\mathsf{RST}^{#1}(\mathcal{P})}
\newcommand{\RSTyf}[1]{\mathsf{RST}_{#1}(\mathcal{P})}
\newcommand{\RSTtyf}[2]{\mathsf{RST}^{#1}_{#2}(\mathcal{P})}
\newcommand{\target}{\mu}
\newcommand{\initial}{\lambda}
\DeclareDocumentCommand{\timeindexed}{ O{t} O{0} m }{(#3_{#1})_{#1 \geq #2}}
\newcommand{\comment}[1]{}
\newcommand{\tentative}[1]{}
\newcommand{\fullstop}{\text{ .}}
\newcommand{\comma}{\text{ ,}}
\comment{(}\newcommand{\lcro}[1]{[#1)}\comment{]}
\comment{[}\newcommand{\lorc}[1]{(#1]}\comment{)}
\newcommand{\Ct}[1][t]{C\left(\lcro{#1,\infty}\right)}
\newcommand{\init}[1]{#1^<}

\newcommand{\Blaw}[2]{\mathbb W^{#1}_{#2}}

\newcommand{\Sgood}[1]{measurable, \(\timeindexed{\Fnull}\)-adapted}

\numberwithin{equation}{section}

\makeatletter
\def\newaliasedtheorem#1[#2]#3{%
  \newaliascnt{#1@alt}{#2}
  \newtheorem{#1}[#1@alt]{#3}
  \expandafter\newcommand\csname #1@altname\endcsname{#3}
}
\makeatother
\newtheorem{theorem}{Theorem}
\numberwithin{theorem}{section}

\newaliasedtheorem{corollary}[theorem]{Corollary}
\newaliasedtheorem{lemma}[theorem]{Lemma}
\newaliasedtheorem{BCHtheorem}[theorem]{Theorem}
\newaliasedtheorem{BCHlemma}[theorem]{Lemma}
\newaliasedtheorem{proposition}[theorem]{Proposition}
\theoremstyle{definition}
\newaliasedtheorem{definition}[theorem]{Definition}
\newtheorem*{problem}{Problem}
\newtheorem{assumption}{Assumption}

\newcommand{\assref}[1]{\ref{ass:specific}.\ref{#1}}

\newcommand{\probref}[1]{{\normalfont (\nameref{#1})}}
\theoremstyle{remark}
\newaliasedtheorem{remark}[theorem]{Remark}

\newcommand{\SG}{\mathsf{SG}}
\newcommand{\tsint}{\scalebox{1.2}{$\int$}}
\newcommand{\tsiint}{\scalebox{1.2}{$\iint$}}
\newcommand{\tsiiint}{\scalebox{1.2}{$\iiint$}}
\newcommand{\conc}{\odot}

\newcommand{\shortandlong}[3][short]{\IfBeginWith{short}{#1}{#2}{\IfBeginWith{long}{#1}{#3}{\IfBeginWith{shortandlong}{#1}{{\color{teal}\textbf{Short Version.} #2} {\color{blue}\textbf{Long Version.} #3}}{\errmessage{First optional argument of shortandlong should be either "short" or "long"}}}}}

\begin{document}
\title{Geometry of Distribution-Constrained Optimal Stopping Problems}
\author{%
Mathias Beiglb\"ock \and Manu Eder \and Christiane Elgert \and Uwe Schmock
}
\thanks{%
  Corresponding author: Mathias Beiglb\"ock, \email{mathias.beiglboeck@tuwien.ac.at}\\
  The first and the second author were supported by FWF-grant  Y00782,  the third author by Jubil\"aumsfonds 16549, the fourth author by Jubil\"aumsfonds 16549 and FWF-grant P25216.
  \\[-0.5ex]\rule{60pt}{0.4pt}\\[0.5ex]
  All authors \\ TU Vienna, Faculty of Mathematics, Wiedner Haupstra\ss e 8-10, 1040 Vienna.
  \\[-0.5ex]\rule{60pt}{0.4pt}\\[0.5ex]
  \gobbletwo
}
\date{\today}

\begin{abstract}
We adapt ideas and concepts developed in optimal transport (and its martingale variant)  to give a geometric description of optimal stopping times \(\tau\) of Brownian motion subject to the constraint that the distribution of \(\tau\) is a given probability \(\target\). The methods work for a large class of cost processes. (At a minimum we need the cost process to be measurable and \(\timeindexed{\Fnull}\)-adapted. Continuity assumptions can be used to guarantee existence of solutions.) We find that for many of the cost processes one can come up with, the solution is given by the first hitting time of a barrier in a suitable phase space. As a by-product we recover classical solutions of the inverse first passage time problem / Shiryaev's problem. 


\subclass{Primary 60G42, 60G44 \and Secondary 91G20}
\end{abstract}

\maketitle

\section{Appetizer}
\label{sec:appetizer}
To whet the reader's appetite and to give some idea of the kind of problems that can be solved with the methods presented in this paper we would like to start with two corollaries to our main results. In \autoref{sec:results} we will present these main results and in \autoref{sec:digesting} we will use them to prove \autoref{cor:Bbarrier} from them.

Both \autoref{cor:Bbarrier} and \autoref{cor:maxbarrier} assert that the solutions of certain optimal stopping problems can be described by a barrier in an appropriate phase space.

In this section, let \(\timeindexed{B}\) be a Brownian motion started\footnote{We note that the results presented in this section remain valid for Brownian motions started according to a general law $\lambda$ at the cost of slightly more tedious moment conditions in the formulation of Corollaries \ref{cor:Bbarrier} and \ref{cor:maxbarrier}.}  in \(0\) on some filtered probability space \( (\Omega,\G,\timeindexed{\G},\P ) \) satisfying the usual conditions and let \(\target\) be a measure on \( (0,\infty) \).
First we consider optimal stopping problems of the following form.

\begin{problem}[\textsc{OptStop}$^{\psi(B_t,t)}$]
\label{OptStoppsiBt}
Among all stopping times \(\tau \sim \target\) on \( (\Omega,\G,\timeindexed{\G},\P ) \) find the maximizer of
\begin{align*}
\tau \mapsto \E[Z_{\tau}]\comma
\end{align*}
where the process $Z$ is of the form \(Z_t = \psi(B_t,t)\).
\end{problem}

\begin{corollary}
\label{cor:Bbarrier}
Assume that \(\target\) has finite first moment. There is an upper semicontinuous function \( \beta : \Time \rightarrow \loneextR \) such that the stopping time
\begin{align}\label{eq:barriertype}
\tau := \inf \left\{ t > 0 : B_t \leq \beta(t) \right\}
\end{align}
has distribution \(\target\).

\(\tau\) has the following uniqueness properties:
On the one hand it is the a.s.\ unique stopping time which has distribution \(\target\) and which is of the form \eqref{eq:barriertype} (we will later say that such a stopping time is the hitting time of a \emph{downwards barrier}).

On the other hand \(\tau\) is also the a.s.\ unique solution of \probref{OptStoppsiBt} for a number of different \(\psi\).
Namely:
\begin{itemize}
\item Let \(p \geq 0\), assume \(\target\) has finite moment of order \(\frac 1 2 + p + \varepsilon\) for some \(\varepsilon > 0\) and let \(A: \Time \rightarrow \R\) be strictly increasing and \(|A(t)| \leq K (1 + t^p)\) for some constant \(K\).%
\footnote{One may of course choose \(0 \leq p < \frac 1 2\), \(\varepsilon := \frac 1 2 - p\) and e.g. \(A(t) := t^p\) so that no moment conditions beyond those at the very beginning of this theorem are imposed on \(\target\).}
Then we may choose
\[ \psi(B_t,t) = B_t A(t) \fullstop \]\vspace{-3mm}
\item Let \(p \geq 2\), assume \(\target\) has finite moment of order \(\frac p 2 + \varepsilon\) for some \(\varepsilon > 0\) and let \(\phi : \R \rightarrow \R\) satisfy \( \phi''' > 0\) as well as \(\abs{\phi(y)} \leq K (1 + |y|^p)\) for some constant \(K\). Then we may choose
\[ \psi(B_t,t) = \phi(B_t) \fullstop \]\vspace{-3mm}
\tentative{\footnote{\tentative{We say \(p \geq 2\) because for \(p < 2\) there are no \(\phi : \R \rightarrow \R\) with \(\phi''' > 0\) and \(\abs{\phi(y)} \leq K (1 + |y|^p)\).}}}%
\end{itemize}
\end{corollary}

To give an example of a slightly more complicated functional amenable to analysis with our tools consider

\begin{problem}[\textsc{OptStop}$^{\rmaxP B t}$]
\label{OptStopmax}
Among all stopping times \(\tau \sim \target\) on \( (\Omega,\G,\timeindexed{\G},\P ) \) find the maximizer of
\begin{align*}
\tau \mapsto \E[\rmaxP{B}{\tau}]\comma
\end{align*}
where \(\rmaxP B t = \sup_{s \leq t} B(s)\).
\end{problem}

\begin{corollary}
\label{cor:maxbarrier}
Assume that \(\target\) has finite moment of order \(\frac{3}{2}\). Then \probref{OptStopmax} has a solution \(\tau\) given by
\begin{align*}
\tau = \inf \left\{ t >0 : B_t - \rmaxP B t \leq \beta(t) \right\}
\end{align*}
for some upper semicontinuous function \(\beta : \Time \rightarrow [-\infty,0]\).%
\end{corollary}
We emphasize that the solutions to the constrained optimal stopping problems provided in Corollaries \ref{cor:Bbarrier} and \ref{cor:maxbarrier} represent particular applications of the abstract results obtained below. Figure \ref{fig:Test} presents graphical depictions of stopping rules of several further solutions of constrained optimal stopping problems (together with the respective optimality properties). These stopping rules can be derived -- under suitable moment conditions -- using arguments very similar to those required for  Corollaries \ref{cor:Bbarrier} and \ref{cor:maxbarrier} (see also the comments  in \autoref{rem:otherpictures} at the end of the paper).
\begin{figure}[h]
\vspace{-2cm}
   \centering
       \includegraphics[page=1,width=0.95\textwidth]{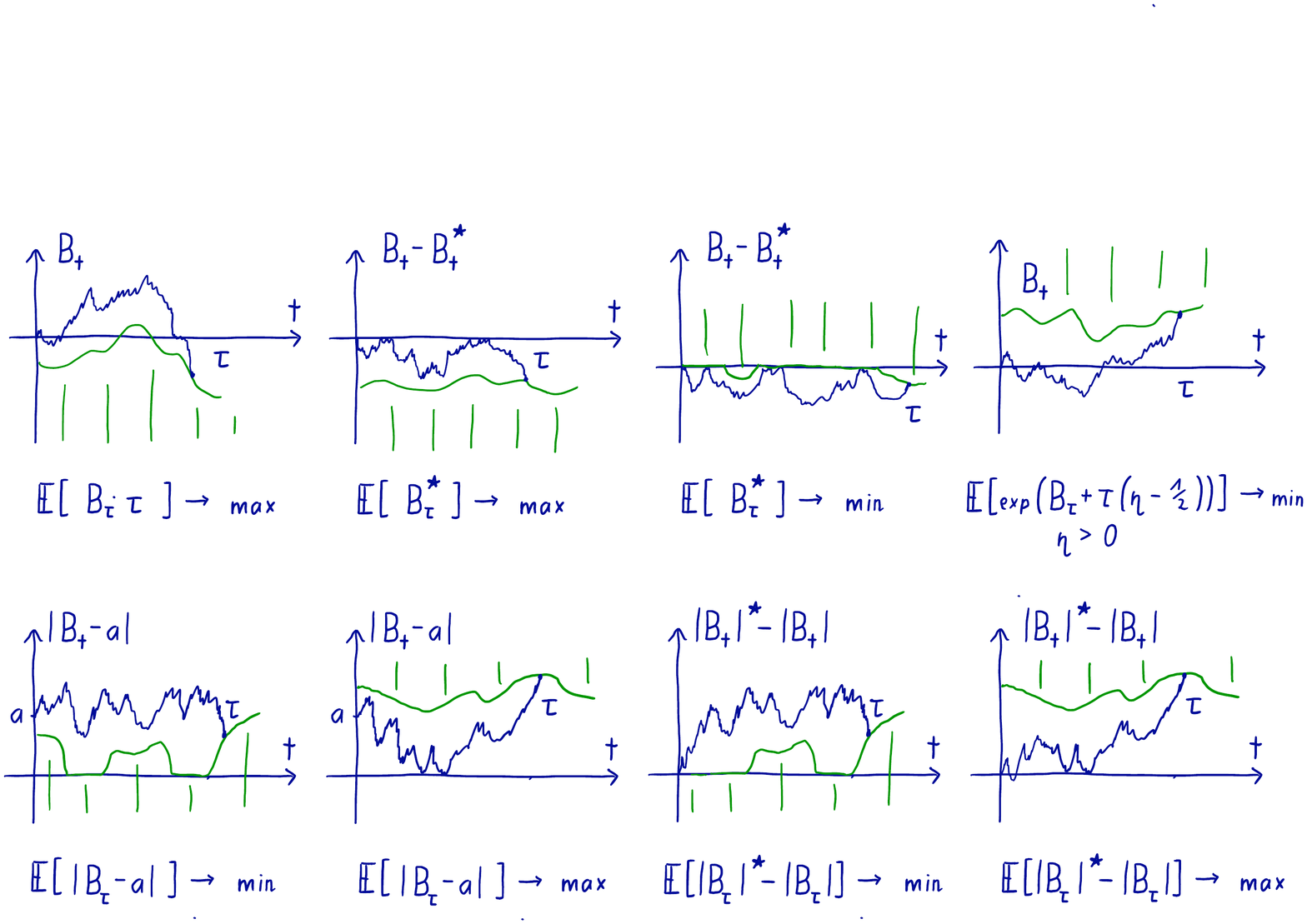} 
 \caption{Solutions to  constrained optimal stopping problems.}
\label{fig:Test}%
\end{figure}

\section{Background - Martingale Optimal Transport and Shiryaev's problem}

In this article we consider distribution-constrained stopping problems from a mass transport perspective. Specifically we find that problems of the form exemplified in  \probref{OptStoppsiBt} and 
\probref{OptStopmax} are amenable to techniques originally developed 
for the martingale version of the classical mass transport problem.
 This martingale optimal transport problem arises naturally in robust finance;  papers to investigate such problems include \cite{HoNe12,BeHePe12,GaHeTo12,DoSo12,CaLaMa14,GhKiLi16b,NuSt16}. In mathematical finance, transport techniques complement the Skorokhod embedding approach (see \cite{Ob04,Ho11} for an overview) to model-independent/robust finance.

A fundamental idea in optimal transport is that the optimality of a transport plan is reflected by the geometry of its support set which can be characterized using the notion of \emph{$c$-cyclical monotonicity}.  The relevance of this concept for the theory of optimal transport has been fully recognized by Gangbo and McCann \cite{GaMc96}, based on earlier work of Knott and Smith \cite{KnSm84} and R\"uschendorf \cite{Ru91,Ru95} among others. Inspired by these ideas, the  literature on martingale optimal transport has developed a `monotonicity principle' which allows to characterize martingale transport plans through geometric properties of their support sets,  cf.\ \cite{BeJu12,Za14,BeGr14,BCH,GuTaTo15b,BeNuTo16}. 

The main contribution of this article is to establish a monotonicity principle which is applicable  to distribution-constrained optimal stopping problems. This transport approach turns out to be remarkably powerful, in particular we will find that questions as raised in Problems \probref{OptStoppsiBt} and 
\probref{OptStopmax} can be addressed using a relatively intuitive set of arguments.

The distribution-constrained optimal stopping problem \probref{OptStop} (and specifically \probref{OptStopmax}) arises naturally in financial and actuarial mathematics. We refer the reader to \cite{Hi13} which describes various  examples (unit-linked life insurances, stochastic modelling for health insurances, the liquidation of an investment portfolio, the valuation of swing options).

Bayraktar and Miller \cite{distconstroptstop} consider the same optimization problem that we treat here. However their setup and methods are rather distinct from the ones used here:  they assume that the target distribution is given by finitely many atoms and that the target functional depends solely on the terminal value of Brownian motion. Following the measure valued martingale approach of Cox and K\"allblad \cite{CoKa15}, \cite{distconstroptstop} address the constrained optimal stopping problem using a Bellman perspective.

The problem to construct a stopping time $\tau$ of Brownian motion such that the law of $\tau$ matches a given distribution on the real line was proposed by Shiryaev in his Banach Center lectures in the 1970's, it has since been called Shiryaev's problem or inverse first passage problem. Dudley and Gutmann \cite{DuGu77} provide an abstract measure-theoretic construction. 
An early barrier-type solution to the inverse first passage problem was given by Anulova \cite{Anu}. She constructs a symmetric two-sided barrier (corresponding to the case $a=0$ in the sixth picture of \autoref{fig:Test}). Anulova discretises the measure $\mu$ and concludes through approximation arguments. 
The solution to the inverse first passage problem given in  Corollary \ref{cor:Bbarrier} was derived by Chen, Cheng,  and Chadam, and Saunders \cite{invfpdiff} based on a variational inequality which describes the corresponding barrier. Notably,  this is  predated by a (formal) PDE description of such barriers given by  Avellaneda and Zhu \cite{AvZh01} in the context of credit risk modeling. Ekstr\"om and Janson \cite{invfpdiff} relate this solution to an optimal stopping problem and provide an integral equation for the barrier. Analytic solutions to the inverse first passage problem are known only in a few cases (\cite{Br67,Le86,Sa88,Pe02,AlPa05,AlPa14}). An interesting connection between the inverse first passage problem and Skorokhod's problem is provided by Jaimungal, Kreinin,  and Valov \cite{JaKrVa14}.

\section{Statement of Main Results}
\label{sec:results}%

\begin{assumption}
\label{ass:general}
Throughout we will assume that \( (\Omega,\G,\timeindexed{\G},\P ) \) is a filtered probability space and that \( \timeindexed{B} \) is an adapted process which has continuous paths on \( ( \Omega,\G,\timeindexed{\G} ) \), such that \( B \) can be regarded as a measurable map from \(\Omega\) to \(\Paths\), the space of continuous functions from \(\Time\) to \(\R\).
The cost function \(\cost\) will always be a measurable map \(\Spacetime \rightarrow \R\).
\(\target\) will denote a probability measure on \(\Time\).
\end{assumption}

Then the problem we consider can be stated as follows.
\begin{problem}[\textsc{OptStop}]
\label{OptStop}
Among all stopping times $\tau \sim \target$ find the minimizer of
\begin{align*} \tau \mapsto \E[\cost(B,\tau)] \fullstop \end{align*}
\end{problem}
Here we formulate our main optimization problem in terms of minimization, following the usual convention in the optimal transport literature (which is also used in the closely related paper \cite{BCH}). Clearly, a sign change transforms this into a maximization problem and in our applications we will in fact  turn to this latter version  when resulting formulations appear more natural. We trust that this will not cause confusion. 

Throughout we will also make the following assumptions without further mention: 
\begin{assumption}
\label{ass:specific}\mbox{}
\begin{enumerate}
\item\label{ass:cadapted}
  \( \cost \) is \Sgood', where \(\timeindexed{\Fnull}\) is the filtration on \(\Paths\) generated by the canonical process \(\left(\omega \mapsto \omega(t)\right)_{t \in \Time{}} \).
\item\label{ass:randomization}
  There is a \( \G_0 \)-measurable random variable \( U \) which is uniformly distributed on \( [0,1] \) and independent of the process \( \timeindexed{B} \).
\item\label{ass:B}
  There is a probability measure $\lambda$ s.t.\ \( \timeindexed{B} \) is a Brownian motion with initial law \(\initial\), i.e.\ \(B_0 \sim \initial\).
\item\label{ass:wellposed}
  The problem is well-posed in the sense that \( \E[\cost(B,\tau)] \) is defined and \( > -\infty\) for all stopping times \( \tau \sim \target \) and that \( \E[\cost(B,\tau)] < \infty \) for at least one such stopping time.
\item\label{ass:finitemoment}
  \( \tsint t^{p_0} \d{\target}(t) < \infty\), where \(p_0 \geq 0\) is some constant that we fix here and that can be chosen when applying the results from this section.
\end{enumerate}
\end{assumption}

A note on language: The adjective \enquote{adapted} is usually applied to processes whose time argument is written in subscript form. For any filtered measurable space \(\tilde\Omega\) and any function \(f : \tilde\Omega \times \Time \rightarrow \R \) (or possibly \(f : \tilde\Omega \times \Time \rightarrow [-\infty,\infty] \)) we will interchangeably think of \(f\) simply as a function or as the process \( Y_t(\omega) := f(\omega,t) \). And so $f$ being adapted means the same thing as \((Y_t)_{t \in \Time}\) being adapted. Similarly for a subset \(\Gamma\) of \(\tilde\Omega \times \Time\) we may also think of \(\Gamma\) as its indicator function or as the process \(Y_t(\omega) := \indicator{\Gamma}(\omega,t)\) and will also say that the set \(\Gamma\) is adapted.

With that in mind, Assumption \assref{ass:cadapted} should seem like an obvious thing to ask for from the cost function. Also, knowing about the existence of optional projections, it should be clear no later than \autoref{lem:equivOpt} that Assumption \assref{ass:cadapted} does not pose a real restriction on the class of problems we are treating.

The role of Assumption \assref{ass:randomization} should become clearer soon.
We would like to note at this point though that often enough our results put together will imply that the solution of Problem \probref{OptStop} for a space \( (\Omega,\G,\timeindexed{\G},\P ) \) which satisfies Assumption \assref{ass:randomization} is essentially the same as the solution of the Problem for a space which may not satisfy said assumption, and we will find that we can describe this solution in detail.
This can be seen executed in the proofs of the corollaries stated in the \nameref{sec:appetizer}.

The methods in this paper work not just for Brownian motion but for a class of processes which is conceptually bigger, but then turns out to not include much beyond Brownian motion -- namely for any space-homogeneous but possibly time-inhomogeneous Markov process with continuous paths which has the strong Markov property. (Here space-homogeneous means that starting the process at location \(x\) and then moving its paths to start at location \(y\) results in a version of the process started at \(y\).) If the reader so wishes, she may think of \(B\) as a process from this slightly larger class of processes. Care was taken not to reference any properties of Brownian motion beyond those stated here. In particular our results apply to multi-dimensional Brownian motion. 

Assumption \assref{ass:wellposed} is mostly just there to ensure that we are actually talking about an optimization problem in a meaningful sense. For the problems presented in the \nameref{sec:appetizer}, the moment conditions on \(\target\) which are given in the statement of \autoref{cor:Bbarrier} and \autoref{cor:maxbarrier} ensure that Assumption \assref{ass:wellposed} is satisfied (as we will see in the proofs of these corollaries).

The constant \(p_0\) in Assumption \assref{ass:finitemoment} will (implicitly) appear in the statement of \autoref{thm:monotonicity}, one of the main results.
Its role is to ensure that \(\E[\varphi(B,\tau)]\) will be finite for some (class of) function(s) \(\varphi\) and any solution \(\tau\) of \probref{OptStop}. (The choice \(\varphi(B,\tau) = \tau^{p_0}\) is somewhat arbitrary here.)

The main results are \autoref{thm:existence} and \autoref{thm:monotonicity}.

We give two versions of \autoref{thm:existence}. Version A is easier to state and may feel more natural, but we will need Version B (which is more general and has essentially the same proof as Version A) in the proof of the corollaries in the Appetizer.

\begin{theorem}
\label{thm:existence}\textcolor{white}{.}

\textbf{Version A.}
Assume that the cost function \(\cost\) is bounded from below and lower semicontinuous when we equip \(\Paths\) with the topology of uniform convergence on compacts.
Then the Problem \probref{OptStop} has a solution.

\textbf{Version B.}
Assume that the cost function \(\cost\) is lower semicontinuous when we equip \(\Spacetime\) with the product topology of two Polish topologies which generate the right sigma-algebras on \(\Paths\) and \(\Time\) respectively and assume that the set
\( \left\{ \cost_-(B,\tau) : \tau \sim \target \text{, $\tau$ is a stopping time} \right\} \) is uniformly integrable, where \(\cost_- := -\cost \vee 0 \) denotes the negative part of \(\cost\).
Then the Problem \probref{OptStop} has a solution.
\end{theorem}

To state \autoref{thm:monotonicity} we need a few more definitions.

\begin{remark}
We will find it convenient to talk about processes that don't start at time \(0\) but instead at some time \( t > 0 \). Similarly we will consider stopping times taking values in \(\lcro{t,\infty}\). These will be defined on the space \(\Paths[t]\) equipped with the filtration \(\timeindexed[s][t]{\F^t}\), again generated by the canonical process \(\left(\omega \mapsto \omega(s)\right)_{s \geq t} \).
We refer to the distribution of Brownian motion started at time \(t\) and location \(x\) by \(\Blaw{t}{x}\). This is a measure on \(\Paths[t]\). For a probability measure \(\kappa\) on \(\R\) we write \(\Blaw{t}{\kappa}\) for the distribution of Brownian motion started at time \(t\) with initial law \(\kappa\).
\end{remark}

\begin{definition}[Concatenation]
\label{def:concat}%
For every \(t \in \Time\) we have an operation \(\conc\) of concatenation, which is a map into \(\Paths[t]\) and is defined for \((\omega,s) \in \Spacetime[t]\) and \(\theta \in \Ct[s]\) with \(\theta(s)=0\) by
\begin{align}
\label{eq:concat}
\left((\omega,s) \conc \theta\right)(r) = 
\begin{cases}
\omega(r) & t\leq r \leq s \\
\omega(s) + \theta(r) & r > s
\end{cases}\fullstop
\end{align}
\end{definition}

\begin{definition}[Stop-Go pairs]
\label{def:SG}
The set of Stop-Go pairs \(\SG \subseteq \left(\Spacetime\right) \times \left(\Spacetime\right)\) is defined as the set of all pairs \( ((\omega,t),(\eta,t)) \) (note that the time components have to match) such that
\begin{align}
\label{eq:SG}
\cost(\omega,t) + \tsint \cost((\eta,t) \conc \theta, \sigma(\theta)) \d{\Blaw{t}{0}}(\theta) < \cost(\eta,t) + \tsint \cost((\omega,t) \conc \theta, \sigma(\theta)) \d{\Blaw{t}{0}}(\theta)
\end{align}
for all \(\timeindexed[s][t]{\F^t}\)-stopping times \(\sigma\) for which \(\Blaw{t}{0}(\sigma = t) < 1\), \(\Blaw{t}{0}(\sigma = \infty) = 0\), \(\tsint \sigma^{p_0} \d{\Blaw{t}{0}} < \infty\) and for which both sides in \eqref{eq:SG} are defined and finite.
\end{definition}

\begin{figure}[h]
\vspace{-6cm}
   \centering
\mbox{$ $}\hspace{-1.5cm}       \includegraphics[page=1,width=1.1\textwidth]{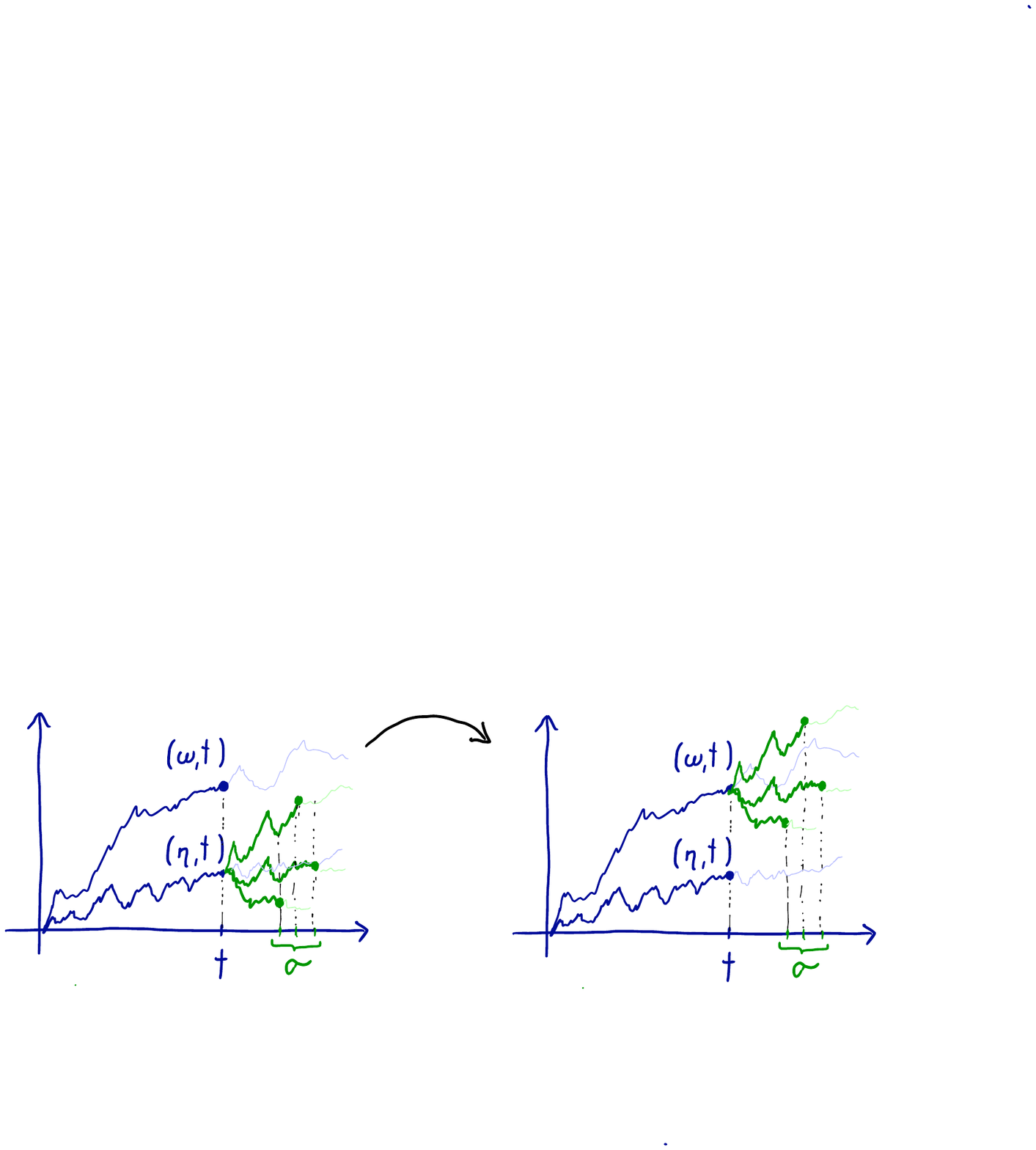} \hspace{-1cm}
       \vspace{-1.5cm}
 \caption{The left hand side of \eqref{eq:SG} corresponds to averaging the function $c$ over the stopped paths on the left picture, the right hand side to averaging the function $c$ over the stopped paths on the right picture.}
 \label{fig:StopGo}
\end{figure}

%\begin{remark}
%\label{rem:SG}
A hopefully intuitive way of putting the definition of Stop-Go pairs into words is the following: $((\omega,s),(\eta,t))$ form a Stop-Go pair iff, irrespective of how we might stop after time $t$ (i.e. which stopping rule $\sigma$ we might use after time $t$), \emph{Stop}ping $\omega$ at time $t$ and letting $\eta$ \emph{Go} on is better -- i.e.\ has lower cost -- than stopping $\eta$ and letting $\omega$ go on.
%\end{remark}

As hinted at earlier, the definition of Stop-Go pairs depends on the parameter \(p_0\) from Assumption \assref{ass:finitemoment}. A larger \(p_0\) means that we are asking for more in Assumption \assref{ass:finitemoment} and implies that we get a larger set \(\SG\), as we are quantifying over fewer stopping times \(\sigma\) in the definition of \(\SG\). This in turn implies that the conclusion of \autoref{thm:monotonicity} below will be stronger.

\begin{definition}[Initial Segments]
\label{def:init}
For a set \(\Gamma \subseteq \Spacetime\) define the set \(\init\Gamma \subseteq \Spacetime\) by
\begin{align}
\label{eq:init}
\init{\Gamma} = \left\{ (\omega,s) : (\omega,t) \in \Gamma \text{ for some } t > s \right\} \fullstop
\end{align}
\end{definition}

\begin{theorem}[Monotonicity Principle]
\label{thm:monotonicity}
Assume that \( \tau \) solves \probref{OptStop}. Then there is a \Sgood' set \( \Gamma \subseteq \Spacetime \) such that \[ \P[ (\timeindexed{B},\tau) \in \Gamma ] = 1 \] and
\begin{align}
\label{eq:monotonicity}
\SG \cap \left(\init{\Gamma} \times \Gamma\right) = \emptyset \fullstop
\end{align}
\end{theorem}

The following lemma should give a first hint about how the \nameref{thm:monotonicity} can be applied.

\newcommand{\taucl}{\underaccent{\check}\tau}
\newcommand{\tauop}{\hat\tau}
\newcommand{\Rcl}{\underaccent{\check}{\mathcal{R}}\vphantom{\hat{\mathcal{R}}}}
\newcommand{\Rop}{\hat{\mathcal{R}}}
\begin{lemma}
\label{lem:geometry}
Let \(\tau\) be a solution of \probref{OptStop} and assume that the cost function \(\cost\) is such that there exists a \Sgood' process \(\timeindexed{Y}\) such that
\begin{align}
\label{eq:yaxis}
Y_t(\omega) < Y_t(\eta) \implies \left((\omega,t),(\eta,t)\right) \in \SG \fullstop
\end{align}
Define the \emph{barriers} \(\Rcl,\Rop \subseteq \R \times \Time\) by
\begin{align*}
\Rcl & = \bigcup_{(\omega,t) \in \Gamma} \lorc{-\infty,Y_t(\omega)} \times \{t\} \\
\Rop & = \bigcup_{(\omega,t) \in \Gamma} \left(-\infty,Y_t(\omega)\right) \times \{t\} \comma
\end{align*}
where \(\Gamma\) is a set with the properties in \autoref{thm:monotonicity}.
Define the functions \(\taucl\) and \(\tauop\) on \(\Paths\) by
\begin{align*}
\taucl(\tilde\omega) & = \inf \left\{ t \in \Time : \left(Y_t(B(\tilde\omega)),t\right) \in \Rcl \right\} \\
\tauop(\tilde\omega) & = \inf \left\{ t \in \Time : \left(Y_t(B(\tilde\omega)),t\right) \in \Rop \right\} \fullstop
\end{align*}
Then
\begin{align}
\label{eq:geometry}
\taucl \leq \tau \leq \tauop \quad \P\text{-a.s.}
\end{align}
\end{lemma}
When applying this Lemma to show that some optimal stopping problem has a barrier-type solution as symbolized for example by the pictures in \autoref{fig:Test} the process $Y_t(B)$ is of course what we are labelling the vertical axes in the pictures with.
So for the first picture $Y_t(\omega) = \omega(t)$, for the second one $Y_t(\omega) = \omega(t) - \sup_{s \leq t}\omega(s)$, for the third $Y_t(\omega) = -(\omega(t) - \sup_{s \leq t}\omega(s))$ (the sign is flipped relative to the labelling in the picture because in this picture the barrier is drawn \enquote{up} instead of \enquote{down}), etc.

Notice that, contrary to customs, when we draw the barriers \(\Rcl\)/\(\Rop\) in the pictures in \autoref{fig:Test} the first coordinate is the vertical axis and the second coordinate is the horizontal axis. This is because, to make cross-referencing and comparison with \cite{BCH} easier, we follow their convention of always having time as the second coordinate but still in the pictures it seems more natural to put the independent variable on the horizontal axis.

Note that a priori \(\taucl\) and \(\tauop\) need not be stopping times or even measurable, as we don't know much about the sets \(\Rcl\) and \(\Rop\).

Using the properties of a concrete process \(\timeindexed{Y}\) we will in the proofs of Corollaries \ref{cor:Bbarrier} and \ref{cor:maxbarrier} be able to show that \(\taucl = \tauop\) a.s.\ (this should not be surprising as for each time \(t\) the barriers \(\Rcl\) and \(\Rop\) differ by at most a single point) and therefore that the optimizer \(\tau\) is the hitting time of a barrier.

\begin{proof}[Proof of Lemma \ref{lem:geometry}]
Let \(\tilde\omega \in \Omega\) s.t.\ \(\left(B(\tilde\omega),\tau(\tilde\omega)\right) \in \Gamma\). By assumption this holds for \(\P\)-a.a.\ \(\tilde\omega\).
Then \(\left(Y_{\tau(\tilde\omega)}(B(\tilde\omega)),\tau(\tilde\omega)\right) \in \Rcl\) and therefore \(\taucl(\tilde\omega) \leq \tau(\tilde\omega)\).

Next we show that \(\tauop(\tilde\omega) \geq \tau(\tilde\omega)\).
Assume that \(\left(Y_t(B(\tilde\omega)),t\right) \in \Rop\). We want to show that \(t \geq \tau(\tilde\omega)\).
By the definition of \(\Rop\) we find that there is \(\eta \in \Paths\) with \((\eta,t) \in \Gamma\) and \(Y_t(B(\tilde\omega)) < Y_t(\eta)\), so by \eqref{eq:yaxis} we know \(\left((B(\tilde\omega),t),(\eta,t)\right) \in \SG\).
Assuming, if possible, \(t < \tau(\tilde\omega)\) we get according to \autoref{def:init} that \((B(\tilde\omega),t) \in \init{\Gamma}\). Therefore we have that \(\left((B(\tilde\omega),t),(\eta,t)\right) \in \SG \cap \left(\init\Gamma \times \Gamma\right)\), but this is a contradiction to \(\SG \cap \left(\init\Gamma \times \Gamma\right) = \emptyset\), so we must have \(t \geq \tau(\tilde\omega)\).
\end{proof}

\begin{remark}[Duality]
Problem \probref{OptStop} is an infinite-dimensional linear programming problem and one would hence expect that a corresponding dual problem can be formulated. Indeed, assuming that $c$ is lower semicontinuous and bounded from below, the value of the optimization problem equals 
\[ \sup_{M,\psi} \E[M_0]+ \tsint \psi\, d\mu, \]
where the supremum  is taken over  bounded $\timeindexed\G$-martingales $M=(M_t)_{t\geq 0}$ and  bounded continuous functions $\psi : \R_+\to \R$ satisfying (up to evanescence)
\[ M_t+\psi(t)\leq c(B,t)\fullstop\]
This can be established in complete analogy to the duality result derived in \cite[Theorem 1.2 / Section 4.2]{BCH}  and we do not elaborate.
\end{remark}

\section{Digesting the Appetizer}
\label{sec:digesting}%

We will now demonstrate how to use the Monotonicity Principle of \autoref{thm:monotonicity} to derive \autoref{cor:Bbarrier}. The proof of \autoref{cor:maxbarrier} is very similar but relies on understanding a technical detail which does not add much to the story at this point, so we leave it for the end of the paper.

\newcommand{\Barrier}{\mathcal{R}}
Both of the sets \(\Rcl\) and \(\Rop\) in \autoref{lem:geometry} have the property that (writing \(\Barrier\) for the set in question) \((y,t) \in \Barrier\) and \(y' \leq y\) implies \((y',t) \in \Barrier\). We call such sets (downwards) barriers. More specifically, for technical reasons in what follows it is slightly more convenient to talk about subsets of \(\extR \times \Time\) instead of subsets of \( \R \times \Time \), giving the following definition.
\begin{definition}
Let \(X\) be a topological space. A \emph{downwards barrier} is a set \(\Barrier \subseteq \extR \times X\) such that \(\{-\infty\} \times X \subseteq \Barrier\) and
\begin{align*}
(y,t) \in \Barrier \text{ and } y' \leq y \text{ implies } (y',t) \in \Barrier
\end{align*}
\end{definition}

Clearly, in \autoref{lem:geometry}, instead of talking about \(\Rcl \subseteq \R \times \Time\), we could have talked about \(\Rcl \cup (\{-\infty\} \times \Time) \subseteq \extR \times \Time\) without anything really changing, and likewise for \(\Rop\).

The reader will easily verify the following lemma.

\begin{lemma}
Let \(X\) be a topological space. There is a bijection between the set of all upper semicontinuous functions \(\beta : X \rightarrow \loneextR\) and the set of all \emph{closed} downwards barriers \(\Barrier \subseteq \extR \times X\) (where closure is to be understood in the product topology). This bijection maps any upper semicontinuous function \(\beta\) to the barrier \(\Barrier\) which is the hypograph of \(\beta\)
\begin{align*}
\Barrier := \left\{ (y,x) : y \leq \beta(x) \right\} \comma
\end{align*}
while the inverse maps a barrier \(\Barrier\) to the function \(\beta\) given by
\begin{align*}
\beta(x) := \sup \left\{ y : (y,x) \in \Barrier \right\} \fullstop
\end{align*}
\end{lemma}

What we will show now, on the way to proving \autoref{cor:Bbarrier} is that the first hitting time after \(0\) of any downwards barrier by Brownian motion is a.s.\ equal to the first hitting time after \(0\) of the closure of that barrier. This serves to both resolve the question whether the times in \autoref{lem:geometry} are stopping times and to show that \(\taucl = \tauop\) a.s.

Let us assume for the rest of this section that \(B\) is actually a Brownian motion started in \(0\).

\begin{lemma}
\label{lem:lessnasty}
Let \(\Barrier\) be a downwards barrier in \(\extR \times \Time\). Let \(\closure{\Barrier}\) be the closure of \(\Barrier\) (in the product topology of the usual topologies on \(\extR\) and \(\Time\)).
Define
\begin{align*}
\tau(\omega) & := \inf \{ t > 0 : (B_t(\omega),t) \in \Barrier \} \\
\closure{\tau}(\omega) & := \inf \{ t > 0 : (B_t(\omega),t) \in \closure\Barrier \} \fullstop
\end{align*}
Then \(\tau = \closure\tau\) a.s.
\end{lemma}
\begin{proof}
\newcommand{\taue}{\closure\tau_{\varepsilon}}
As \(\closure\Barrier \supseteq \Barrier\) we clearly have \(\closure\tau(\omega) \leq \tau(\omega)\) for all \(\omega \in \Omega\). Define
\begin{align*}
\taue(\omega) := \inf \{ t > 0 : (B_t(\omega) + \varepsilon \cdot A(t),t) \in \closure\Barrier \} \comma
\end{align*}

where \(A(t) := \frac{t}{1+t} \) is a bounded, strictly increasing function.
Using just that \(\closure\Barrier\) is the closure of \(\Barrier\) one proves by elementary methods that \(\tau(\omega) \leq \taue(\omega)\) for all \(\omega \in \Omega\) and any \(\varepsilon > 0\).
\shortandlong{%
Because \(A(t) = \int_0^t (1+s)^{-2} \d s \) is the integral from \(0\) to \(t\) of a square integrable function we can apply Girsanov's theorem (see e.g.\ \cite[Theorem 38.5]{RoWi00}) to see that \(\closure\tau_{1/n}\) converges to \(\closure\tau\) in distribution as \(n \to \infty\).
}{%

Because \(A(t) = \int_0^t (1+s)^{-2} \d s \) is the integral from \(0\) to \(t\) of a square integrable function we can apply Girsanov's theorem (see e.g.\ \cite[Theorem 38.5 (ii)]{RoWi00}) to see that \(\closure\tau_{1/n}\) converges to \(\closure\tau\) in distribution as \(n \to \infty\).

In detail:

\emph{Route 1:}
To be able to apply \cite[Theorem 38.5 (ii)]{RoWi00} to \(\gamma^\varepsilon_t := - \varepsilon (1+t)^{-2}\) we need to show that \[\zeta^\varepsilon_t := \exp \left( \int_0^t \gamma^\varepsilon_s \d B_s - \frac 1 2 \int_0^t (\gamma^\varepsilon_s)^2 \d s \right)\] are uniformly integrable (they define a martingale, as stated in \cite[Corollary 37.11]{RoWi00}). The integral \(\int_0^t \gamma^\varepsilon_s \d B_s\) is normally distributed with variance \(\int_0^t (\gamma^\varepsilon_s)^2 \d s\) (source: Julio). This allows us to calculate \[\E\left[\abs{\zeta^\varepsilon_t}^{1+\delta}\right] = \exp\left( \frac 1 2 ((1+\delta)^2 - (1+\delta)) \int_0^t (\gamma^\varepsilon_s)^2 \d s \right)\] (or perhaps some other expression that's bounded in \(t\), if I made mistakes while calculating this), which shows that \(\zeta^\varepsilon_t\) are bounded in \(L^{1+\delta}\)-norm, for all \(\delta > 0\) (we really only needed one), uniformly in \(t\), which implies that they are uniformly integrable.

Now we know that \(B_t - \int_0^t \gamma^\varepsilon_s \d s = B_t + \varepsilon \cdot A(t)\) is a Brownian motion wrt the measure \(\mathbb Q^\varepsilon := F \mapsto \tsint F \cdot \zeta^\varepsilon_\infty \d\P\), i.e.\
\begin{align*}
\push{(\omega \mapsto (B_t(\omega) + \varepsilon \cdot A(t))_{t \in \Time})}{\mathbb Q^\varepsilon} = \push{(\omega \mapsto (B_t(\omega))_{t \in \Time})}{\P} \fullstop
\end{align*}

For continuous bounded \(F : \Time \rightarrow \R\)
\begin{multline*}
\lim_{n \to \infty} \tsint F(\closure\tau_{1/n}) \d\P = 
\tsint \lim_{n \to \infty} F(\closure\tau_{1/n}) \d\P = 
\tsint \lim_{n \to \infty} F(\closure\tau_{1/n}) \zeta^{1/n}_\infty \d\P = \\
\lim_{n \to \infty} \tsint F(\closure\tau_{1/n}) \zeta^{1/n}_\infty \d\P = 
\lim_{n \to \infty} \tsint F(\closure\tau_{1/n}) \d \mathbb Q^{1/n} = 
\lim_{n \to \infty} \tsint F(\closure\tau) \d\P \fullstop
\end{multline*}
Here we used that the limit \(\lim_{n \to \infty} \closure\tau_{1/n}\) and therefore also \(\lim_{n \to \infty} F(\closure\tau_{1/n})\) exists pointwise, then that \(\lim_{n \to \infty} \zeta^{1/n}_\infty = 1\) a.s. and in the last equality we used the previous equation.

\newcommand{\Wiener}{\mathbb W}%
\emph{Route 2:}
Set
\newcommand\ttaue[1][\varepsilon]{\tilde\tau_{#1}}
\begin{align*}
\ttaue & : \Paths \rightarrow \Time \\
\ttaue (\eta) & :=  \inf \{ t > 0 : (\eta(t) + \varepsilon \cdot A(t),t) \in \closure\Barrier \} \comma
\end{align*}
s.t.\ \(\taue = \ttaue(B)\).

Set
\newcommand{\Phie}[1][\varepsilon]{\Phi^{#1}}
\begin{align*}
\Phie & : \Paths \rightarrow \Paths \\
\Phie(\eta) & := t \mapsto \eta(t) - \varepsilon \cdot A(t)
\end{align*}
and observe \(\ttaue \circ \Phie = \ttaue[0]\).

Set \(\gamma^\varepsilon_t := \varepsilon (1+t)^{-2}\) and set
\[\zeta^\varepsilon_t := \exp \left( \int_0^t \gamma^\varepsilon_s \d \tilde B_s - \frac 1 2 \int_0^t (\gamma^\varepsilon_s)^2 \d s \right)\]
where \(\tilde B_s(\omega) = \omega(s)\) defines a Brownian motion on \(\Paths\) wrt \(\Wiener\).
The same arguments as in Route 1 show that \(\mathbb Q^\varepsilon := F \mapsto \tsint F \cdot \zeta^\varepsilon_\infty \d\Wiener \) is a probability measure and that \(\push{\Phie}{\mathbb Q} = \Wiener\).

Then for any bounded measurable \(F : \Time \rightarrow \R\):
\begin{multline*}
\tsint F(\taue) \d\P =
\tsint F(\ttaue(\omega)) \d\Wiener(\omega) =
\tsint F(\ttaue(\Phie(\omega))) \d\mathbb Q(\omega) = \\
\tsint F(\ttaue[0](\omega)) \d\mathbb Q(\omega) =
\tsint F(\ttaue[0](\omega)) \zeta^\varepsilon_\infty(\omega)  \d\Wiener(\omega)
\end{multline*}
The last integrand converges pointwise to \(F(\ttaue[0](\omega))\) for \(\varepsilon \to 0\).
Therefore
\begin{align*}
\lim_{n \to \infty} \tsint F(\taue) \d\P = \tsint F(\ttaue[0](\omega)) \d\Wiener(\omega) = \tsint F(\closure\tau) \d\P \fullstop
\end{align*}

}%

\shortandlong{%
As \(\left(\closure\tau_{1/n}\right)_n\) is a decreasing sequence bounded below by \(\closure\tau\) we get that convergence holds almost surely.
}{%
As \(\left(\closure\tau_{1/n}\right)_n\) is a decreasing sequence bounded below by \(\closure\tau\) we get that convergence holds almost surely.
In more detail, as \(\closure\tau_{1/n}\) is a decreasing sequence bounded below by \(\closure\tau\), we know that the pointwise limit exists and also that
\begin{multline*}
\left\{ \lim_{n \to \infty} \closure\tau_{1/n} \neq \closure\tau \right\} = 
\left\{ \lim_{n \to \infty} \closure\tau_{1/n} > \closure\tau \right\} = \\
\bigcup_{q \in \mathbb{Q}_+} \left\{ \lim_{n \to \infty} \closure\tau_{1/n} \geq q \right\} \cap \left\{ q > \closure\tau \right\} = 
\bigcup_{q \in \mathbb{Q}_+} \left\{ \lim_{n \to \infty} \closure\tau_{1/n} \geq q \right\} \setminus \left\{ \closure\tau \geq q \right\}
\end{multline*}
\(\closure\tau \leq \closure\tau_{1/n}\) for all \(n\) implies \( \left\{ \closure\tau \geq q \right\} \subseteq \left\{ \lim_{n \to \infty} \closure\tau_{1/n} \geq q \right\} \) and therefore \[ \P\left[ \left\{ \lim_{n \to \infty} \closure\tau_{1/n} \geq q \right\} \setminus \left\{ \closure\tau \geq q \right\} \right] = \P\left[ \lim_{n \to \infty} \closure\tau_{1/n} \geq q \right] - \P\left[\closure\tau \geq q\right] \] which is \(0\) because \(\closure\tau_{1/n}\) converges to \(\closure\tau\) in distribution and therefore
\begin{multline*}
\P\left[ \lim_{n \to \infty} \closure\tau_{1/n} \geq q \right] = \P\left[ \inf_n \closure\tau_{1/n} \geq q \right] = \P\left[ \bigcap_n \left\{ \closure\tau_{1/n} \geq q \right\} \right] \\ = \lim_{n \to \infty} \P\left[ \closure\tau_{1/n} \geq q \right] \leq \P\left[ \closure\tau \geq q\right] \fullstop
\end{multline*}
}
\end{proof}

The following is a particular case of \cite[Corollary 2.3]{Gr16} (which in turn relies on arguments given in \cite{Ro69,Lo70}). Note that this lemma is purely a statement about barrier-type stopping times and is not directly connected to the optimization problem under consideration.
\begin{lemma}[Uniqueness of Barrier-type solutions]
\label{lem:loynes}
Assume that \( \timeindexed{Y} \) is a \Sgood' process and that the process $Z$ defined through $Z_t:=Y_t(B)$ has a.s.\ continuous paths. 
Let \( \Barrier_1, \Barrier_2 \subseteq \extR \times \Time \) be closed downwards barriers such that for
\begin{align*}
\tau_i(\omega) := \inf \left\{ t > 0 : (Z_t(\omega),t) \in \Barrier_i \right\} 
\end{align*}
we have \(\tau_1 \sim \tau_2\).
Then $\tau_1= \tau_2$ a.s.
\end{lemma}
\shortandlong{%
\begin{proof} Is to be found in \cite[Corollary 2.3]{Gr16}.\end{proof}
}{%
\begin{proof}
I would like to record one detail here for myself in case I wonder about it later:
The context is general starting law and the stopping times may be distributed anywhere on \([0,\infty]\).

So, the strategy is: we show that if \(\tau_1 \sim \target\) and \(\tau_2 \sim \target\), where \(\target\) is a probability on \([0,\infty]\), then \(\tau_3 \sim \target\) where \(\tau_3\) is the hitting time of \(\Barrier_3 := \Barrier_1 \cup \Barrier_2\) by \(Z\). By definition \(\tau_3 \leq \tau_1\) and \(\tau_3 \leq \tau_2\) and because it has the same distribution, it must be equal to both.

To show that \(\tau_3 \sim \target\) we show that \(\P[\tau_3 \in D] \leq \target(D)\) for all \(D \subseteq [0,\infty]\). We do this by finding sets \(D_0, D_1, D_2, D_\infty \subseteq [0,\infty]\) such that \( D_0 \cup D_1 \cup D_2 \cup D_\infty = [0,\infty]\) and showing separately for \(i \in \{0,1,2,\infty\}\) that for \(D \subseteq D_i\) we have \(\P[\tau_3 \in D] \leq \target(D)\).

Define \(D_0 := \{0\}\), \(D_1 := \{ t \in (0,\infty) : (x,t) \in \Barrier_2 \implies (x,t) \in \Barrier_1 \}\), \(D_2 := \{ t \in (0,\infty) : (x,t) \in \Barrier_1 \implies (x,t) \in \Barrier_2 \}\), \(D_\infty = \{\infty\}\).
\(D_1\) is the set where \(\Barrier_1\) is higher than \(\Barrier_2\) and \(D_2\) is the set where \(\Barrier_2\) is higher than \(\Barrier_1\).

The classical part: For a.a.\ \(\omega \in \Omega\) we have: For \(D \subseteq D_1\), if \(\tau_3(\omega) = t \in D\), then \(Z_t(\omega) \in \Barrier_3\) (because \(\Barrier_3\) is closed and \(Z\) has a.s.\ continuous paths) and by the definition of \(D_1\) also \(B_t(\omega) \in \Barrier_1\), so \(\tau_1(\omega) \leq t\). But \(t = \tau_3(\omega) \leq \tau_1(\omega)\), so \(\tau_1(\omega) = t \in D\). Therefore \(\P[\tau_3 \in D] \leq \P[\tau_1 \in D] = \target(D)\). For \(D_2\) swap the roles of \(\tau_1\) and \(\tau_2\), etc.

As \(\tau_3 \leq \tau_1\) we have \(\P[\tau_3 = \infty] \leq \P[\tau_1 = \infty]\). This takes care of \(D \subseteq D_\infty\).

Now we want to show that \(\P[\tau_3 = 0] \leq \target(\{0\})\).
By Blumenthal's 0-1-law, for any \(x \in \R\) the probability that \(Y\) applied to Brownian motion started in \(x\) is immediately stopped by a barrier is either \(0\) or \(1\), so we may define \(\Barrier^0_i := \{ x \in \R : \text{Brownian motion started in $x$ is a.s.\ stopped immediately} \}\) (for \(i \in \{1,2,3\}\)).
Then we have \(\P[\tau_i = 0] = \initial(\Barrier^0_i)\).
Note that \(\Barrier^0_i\) is downwards closed, i.e.\ \(x \in \Barrier^0_i\) and \( y \leq x \) implies \(y \in \Barrier^0_i\).
\( x \notin \Barrier^0_1\) and \( x \notin \Barrier^0_2 \) means that almost surely \(Y\) applied to Brownian motion started in \(x\) will run for a positive amount of time before it hits either \(\Barrier_1\) or \(\Barrier_2\). So it will also run for a positive amount of time before it hits \(\Barrier_3 = \Barrier_1 \cup \Barrier_2\). In other words \(x \notin \Barrier^0_3\).
As \(\Barrier^0_1\) and \(\Barrier^0_2\) are downwards closed, at least one of them contains the other. Without loss of generality, \(\Barrier^0_1 \supseteq \Barrier^0_2\), so \( x \notin \Barrier^0_1 \) implies \( x \notin \Barrier^0_2 \) and these together imply \( x \notin \Barrier^0_3\), so \(1-\target(\{0\}) = \P[\tau_1 > 0] = \initial(\compl{(\Barrier^0_1)}) \leq \initial(\compl{(\Barrier^0_3)}) = \P[\tau_3 > 0] = 1-\P[\tau_3=0] \).
\end{proof}%
}%

We now have the necessary prerequisites to use our main results in showing that the first optimization problem in the \nameref{sec:appetizer} admits a (unique) barrier-type solution.

\begin{proof}[Proof of \autoref{cor:Bbarrier}]
The strategy is as follows: We choose a cost function and leverage \autoref{thm:existence} to show that an optimizer exists, the \nameref{thm:monotonicity} in the form of \autoref{thm:monotonicity} and \autoref{lem:geometry} will -- with some help from \autoref{lem:lessnasty} -- show that any optimizer must be the hitting time of a barrier.
\autoref{lem:loynes} shows that any two barrier-type solutions must be equal.

We now provide the details.
Start with a cost function \(\cost(\omega,t) := -\omega(t) A(t)\) for a strictly monotone function \(A: \Time \rightarrow \R\) which satisfies \(|A(t)| \leq K(1+t^p)\) and assume that \(\target\) has moment of order \(\frac 1 2 + p + \varepsilon\) for some \(\varepsilon > 0\). To prove that a barrier-type solution exists when \(\target\) has first moment, choose a bounded strictly increasing \(A\) and \(p=0\), \(\varepsilon=\frac 1 2\) in this step. (These assumptions guarantee in particular that the optimization problems considered below have a finite value.)
Clearly the problem \probref{OptStop} for \(c\) corresponds to \probref{OptStoppsiBt} for \(\psi(B_t,t) = B_t A(t)\) (i.e.\ $\psi$ takes the role of $-c$ such that the minimal/maximal values agree up to a change of sign).
We will deal with the case where \(\psi(B_t,t) = \phi(B_t)\) at the end of this proof.

We now check that the conditions in Version B of \autoref{thm:existence} are satisfied. We also need to check that \autoref{ass:specific} holds.
Here we need the assumption that \(\target\) has moment of order \(\frac 1 2 + p + \varepsilon\), as well as the H\"older and Burkholder-Davis-Gundy inequalities. The latter specialized to Brownian motion state that for all \(q > 0\) there are positive constants \(K_0\) and \(K_1\) such that for any stopping time \(\tau\) we have \( K_0 \, \E\left[\tau^{q/2}\right] \leq \E\left[(|B|^*_\tau)^q\right] \leq K_1 \, \E\left[\tau^{q/2}\right] \) (where \(|B|^*_t = \sup_{s \leq t} |B_s| \)).
With these in hand a straightforward calculation allows us to bound \(B_\tau A(\tau)\) in the \(L^{1+\delta}\)-norm for some \(\delta > 0\), independently of the stopping time \(\tau \sim \target\).

\shortandlong{%
}{%
In detail, we may set \(q_1 := 1+2p\), \(q_2 := \frac{1+2p}{2p}\), choose \(\delta>0\) s.t. \((1+\delta) (\frac 1 2 + p) = \frac 1 2 + p + \varepsilon\) and calculate.
\begin{multline*}
\E[|B_\tau A(\tau)|^{1+\delta}] \leq
\E[|B_\tau|^{(1+\delta) q_1}]^{1/q_1} \E[|A(\tau)|^{(1+\delta) q_2}]^{1/q_2} \leq \\
K_1 \E[\tau^{(1+\delta) q_1/2}]^{1/q_1} (\E[|1+\tau^p|^{(1+\delta) q_2}]^{1/((1+\delta)q_2)})^{1+\delta} \leq \\
K_1 \E[\tau^{(1+\delta) q_1/2}]^{1/q_1} (1 + \E[\tau^{p (1+\delta) q_2}]^{1/((1+\delta)q_2)})^{1+\delta} = \\
K_1 \E[\tau^{(1+\delta) (1/2+p)}]^{1/q_1} (1 + \E[\tau^{(1+\delta) (\frac 1 2 +p)}]^{1/((1+\delta)q_2)})^{1+\delta} = K_2
\end{multline*}
}%
This shows both that the uniform integrability condition in Version B of \autoref{thm:existence} is satisfied and that Assumption \assref{ass:wellposed} is satisfied.

On \(\Paths\) we may choose the (Polish) topology of uniform convergence on compacts. For the topology on \(\Time\) we start with the usual topology and turn \(A\) into a continuous function (if it wasn't), by making use of the fact that any measurable function from a Polish space to a second countable space may be turned into a continuous function by passing to a larger Polish topology (with the same Borel sets) on the domain. (This can be found for example in \cite[Theorem 13.11]{Ke95}.)

In the statement of \autoref{cor:Bbarrier} we did not require that the probability space \( (\Omega,\G,\timeindexed{\G},\P ) \) satisfy Assumption \assref{ass:randomization}. 
To remedy this we can enlarge the probability space by setting \(\tilde\Omega := \Omega \times [0,1]\), \(\tilde G_t := \G_t \atimes \Borel{[0,1]}\) and \(\tilde\P := \P \mtimes \Lebesgue\), where \(\Lebesgue\) is Lebesgue measure on \([0,1]\). 
On this space we consider the Brownian motion \(\tilde B_t (\omega, x) := B_t(\omega)\). 
\autoref{thm:existence} now gives us an optimal stopping time \(\tilde\tau\) on the enlarged probability space. 
If we can show that this stopping time is in fact the hitting time of a barrier, then it follows that \(\tilde\tau = \tau \circ ((\omega,x) \mapsto \omega)\) for a stopping time \(\tau\) which is defined as the hitting time of the Brownian motion \(B\) of the same barrier. 
As there are \emph{more} stopping times on \((\tilde\Omega, \tilde\G, \timeindexed{\tilde\G}) \) than on  \( (\Omega,\G,\timeindexed{\G} ) \) in the sense that any stopping time \(\tau'\) on \( (\Omega,\G,\timeindexed{\G} ) \) induces a stopping time \(\tilde\tau' := \tau' \circ ((\omega,x) \mapsto \omega)\) on \((\tilde\Omega, \tilde\G, \timeindexed{\tilde\G}) \) we conclude that \(\tau\) must also be optimal among the stopping times on \( (\Omega,\G,\timeindexed{\G} ) \).
With this out of the way, let us refer to our Brownian motion by \(B\), to the optimal stopping time by \(\tau\) and to our filtered probability space by \( (\Omega,\G,\timeindexed{\G},\P ) \) irrespective of whether this is the original process and space we started with, or an enlarged one.

Choosing \( p_0 := \frac 1 2 + p + \varepsilon \) in Assumption \assref{ass:finitemoment} we apply \autoref{thm:monotonicity} to obtain a set \(\Gamma\) on which \((B,\tau)\) is concentrated under \(\P\) and for which \eqref{eq:monotonicity} holds.
As \(\target\) is concentrated on \((0,\infty)\), we may assume that \(\Gamma \cap (\Paths \times \{ 0 \}) = \emptyset\).
Next we want to show that \autoref{lem:geometry} applies with \(Y_t(\omega) = \omega(t)\).

Translating \eqref{eq:yaxis} to our situation, we want to prove that \(\omega(t) < \eta(t)\) implies
\begin{align}
\label{eq:Byaxis}
-\omega(t) A(t) - \E\left[ \left(\eta(t) + \tilde{B}_\sigma\right) A(\sigma) \right] < -\eta(t) A(t) - \E\left[ \left(\omega(t) + \tilde{B}_\sigma\right) A(\sigma) \right] \comma
\end{align}
where \(\tilde B\) is Brownian motion started in \(0\) at time \(t\) on \(\Paths[t]\) and \(\sigma\) is any stopping time thereon with \(\Blaw{t}{0}(\sigma = t) < 1\), \(\Blaw{t}{0}(\sigma = \infty) = 0\), \(\tsint \sigma^{p_0} \d{\Blaw{t}{0}} < \infty\).
Again the Burkholder-Davis-Gundy inequality shows that \(\E[\tilde{B}_\sigma  A(\sigma)] < \infty\).
So \eqref{eq:Byaxis} turns into
\begin{align*}
\omega(t)\, \E[A(\sigma) - A(t)] < \eta(t)\, \E[A(\sigma) - A(t)]
\end{align*}
which clearly follows from the assumptions.
So we know that \autoref{lem:geometry} holds, i.e.\ using the names from said lemma we have \(\taucl \leq \tau \leq \tauop\) \(\P\)-a.s.

\(\Gamma \cap (\Paths \times \{ 0 \}) = \emptyset\) implies \(\Rcl \cap (\R \times \{0\}) = \emptyset\) and therefore \( \taucl(\omega) = \inf \{ t > 0 : (B_t(\omega),t) \in \Rcl \} \), and likewise for \(\Rop\) and \(\tauop\).
As \(\closure\Rcl = \closure\Rop =: \closure\Barrier\) it follows from \autoref{lem:lessnasty} that \(\taucl = \tau = \tauop\) a.s.\ and that $\tau$ is of the form claimed in  \eqref{eq:barriertype} with $\beta(t) := \sup\{y\in \R : (y,t) \in \closure\Barrier\}$.
The uniqueness claims follow from \autoref{lem:loynes} and what we have already proven.

We now treat the case where \(\psi(B_t,t) = \phi(B_t)\) with \(\phi''' > 0\), \(\abs{\phi(y)} \leq K (1 + |y|^p)\) and \(\target\) has finite moment of order \( \frac p 2 + \varepsilon \) for some \(\varepsilon > 0\). Most of the proof remains unchanged.
Setting \(\cost(\omega,t) = -\phi(\omega(t))\) we may again use the Burkholder-Davis-Gundy inequalities to show that \(\cost(B_\tau,\tau)\) is bounded in \(L^{1+\delta}\)-norm, independently of the stopping time \(\tau \sim \target\), thereby showing both that Assumption \assref{ass:wellposed} is satisfied and that the uniform-integrability condition in Version B of \autoref{thm:existence} is satisfied.

It remains to show that \(\omega(t) < \eta(t)\) implies \(((\omega,t),(\eta,t)) \in \SG\).
\(\phi''' > 0\) implies that the map \(y \mapsto \phi(\eta(t) + y) - \phi(\omega(t) + y)\) is strictly convex. By the strict Jensen inequality \(\E[ \phi(\eta(t) + \tilde B_\sigma) - \phi(\omega(t) + \tilde B_\sigma) ] > \phi(\eta(t)) - \phi(\omega(t))\) for any stopping time \(\sigma\) on \(\Paths[t]\) which is almost surely finite, satisfies optional stopping and is not almost surely equal to \(t\).
As we may choose \(p_0 := \frac p 2 + \varepsilon \), which is greater than \(1\), we may assume that the \(\sigma\) in the definition of \(\SG\) has finite first moment, which is enough to guarantee that it satisfies optional stopping.
Rearranging the last inequality gives \eqref{eq:SG}.
\end{proof}

\section{Existence of an Optimizer}
The proof of existence of solutions to the Problem \probref{OptStop} crucially depends on thinking of stopping times as the joint distribution of the process to be stopped and the stopping time. We introduce some concepts to make this precise and give a proof of \autoref{thm:existence} at the end of this section.

\begin{lemma}
\label{lem:cExp}
Let \(G: \Paths[t] \rightarrow \R \), and \(s \geq t\). The function
\begin{align*}
\omega \mapsto \tsint G((\omega,s) \conc \theta) \d{\Blaw{s}{0}}(\theta)
\end{align*}
is a version of the conditional expectation \(\E_{\Blaw{t}{\initial}}[G|\F^t_s]\) (for any initial distribution \(\initial\)). Henceforth, by \(\cExp[t]{s}{G}\) we will mean this function.

If \(G \in C_b\left(\Paths[t]\right)\), then \(\cExp[t]{s}{G} \in C_b\left(\Paths[t]\right)\).
\end{lemma}
\begin{proof}
Obvious.
\end{proof}
Here we use \(C_b(X)\) to denote the set of continuous bounded functions from a topological space \(X\) to \(\R\).
The last sentence of the lemma is of course true for any topology on \(\Paths[t]\) for which the map \(\omega \mapsto \omega \conc \theta\) is continuous for all \(\theta\), but we will only need it for the topology of uniform convergence on compacts.\footnote{And that choice is rather arbitrary itself, as close reading will reveal.}

Given spaces \(X\) and \(Y\) we will denote the projection from \(X \times Y\) to \(X\) by \(\proj{X}\) (and similarly for \(Y\)). For a measurable map \(F: X \rightarrow Y\) between measure spaces and a measure \(\nu\) on \(X\) we denote the pushforward of \(\nu\) under \(F\) by \(\push{F}{\nu} := D \mapsto \nu(\invimage{F}{D})\).

{%
\renewcommand{\initial}{\kappa}%
\renewcommand{\target}{\nu}%
\begin{definition}[$\RST$]
\label{def:RST}%
The set \(\RSTty{t}{\initial}\) of \emph{randomized stopping} times (of Brownian motion started at time \(t\) with initial distribution \(\initial\)) is defined as the set of all subprobability measures \(\xi\) on \(\Spacetime[t]\) such that \(\push{(\proj{\Paths[t]})}{\xi} \leq \Blaw{t}{\initial}\) and that
\begin{align}
\label{eq:RSTadapted}
 \tsint F(r) \left( G(\omega) - \cExp[t]{s}{G}(\omega) \right ) \d{\xi}(\omega,r) = 0
\end{align}
for all \(s > t\), all \( G \in C_b\left(C\left(\lcro{t,\infty}\right)\right) \) and all \(F \in C_b\left(\lcro{t,\infty}\right) \) supported on \([t,s]\).

In this definition the topology on \(\Paths[t]\) is that of uniform convergence on compacts and the topology on \(\Time[t]\) is the usual topology.

Given a distribution \(\target\) on \(C\left(\lcro{t,\infty}\right)\) we write 
\begin{align*}
\RSTtyd{t}{\initial}{\target} := \left\{ \xi \in \RSTty{t}{\initial} : \push{(\proj{\Time[t]})}{\xi} = \target \right\} \fullstop
\end{align*}
We write \(\RSTtyf{t}{\initial}\) for the set of all \(\xi \in \RSTty{t}{\initial}\) with mass \(1\) and call these the \emph{finite} randomized stopping times.

In any of these, if we drop the superscript \(t\) then we will mean time \(t = 0\), while, if we drop the subscript \(\initial\), then we mean that the initial distribution \(\initial = \delta_0\), i.e.\ the Brownian motion to be stopped is started deterministically in \(0\).
\end{definition}%

To explain the qualifier \emph{finite} it may help to imagine that for a non-finite randomized stopping time of mass \(\alpha < 1\), the mass \(1-\alpha\) which is missing is placed along \(C(\lcro{t,\infty})\times\{\infty\}\).

The following \autoref{lem:equivOpt} from \cite{BCH} shows that the problem \probref{OptStop} is equivalent to the following optimization problem \probref{OptStop'} in the sense that a solution of one can be translated into a solution of the other and vice versa. This of course also implies that the values of the two problems are equal, thereby showing that the concrete space \( (\Omega,\G,\timeindexed{\G},\P ) \) has no bearing on this value, as long as Assumptions \ref{ass:general} and \ref{ass:specific} are satisfied.

The definition we have given for a randomized stopping time is only the most convenient (for our purposes) of a number of possible equivalent definitions. Although \autoref{lem:equivOpt} below should provide some intuition on what a randomized stopping time \emph{is}, the reader may still wish to refer to \cite[\autoref{BCH-thm:equiv RST}]{BCH} for the other possible ways of defining randomized stopping times. The first step in connecting condition \eqref{eq:RSTadapted}, which is one of the equivalent conditions listen in said theorem, to the others, is to notice that \eqref{eq:RSTadapted} can be rewritten as
\begin{align*}
 \tsint \left( \tsint F(r) \d{\xi_\omega}(r) \right) \left( G(\omega) - \cExp[t]{s}{G}(\omega) \right ) \d{\Blaw{t}{\initial}}(\omega) = 0 \comma
\end{align*}
where \(\xi_\omega\) is a disintegration of \(\xi\) with respect to \(\Blaw{t}{\kappa}\).
This says that the function \(\omega \mapsto \tsint F(r) \d{\xi_\omega}(r)\) is orthogonal to \(G - \cExp[t]{s}{G} \) for all bounded continuous \(G\), i.e.\ that it is a.s.\ \(\F^t_s\)-measurable whenever \(F\) is supported on \([t,s]\). A limit argument then shows that \(\omega \mapsto \xi_\omega([t,s])\) is a.s.\ \(\F^t_s\)-measurable.
Again, we refer the reader to \cite{BCH} for a more detailed exposition.
}

\begin{problem}[\textsc{OptStop'}]
\label{OptStop'}
Among all randomized stopping times $\xi \in \RSTyd{\initial}{\target}$ find the minimizer of
\begin{align*} \xi' \mapsto \tsint \cost \d{\xi'} \fullstop \end{align*}
\end{problem}

\begin{BCHlemma}[{\cite[Lemma \ref{BCH-lem:equivOpt}]{BCH}}]\label{lem:equivOpt}
  Let $\tau$ be a $\timeindexed\G$-stopping time and consider
  \begin{align*}
    \Phi & : \Omega\to \Paths \times [0,\infty] \\
    \Phi(\omega) & := ((B_t(\omega))_{t\ge 0}, \tau(\omega)) \fullstop
  \end{align*}
  Then $\xi:= \mrestr{\push{\Phi}{\P}}{\Spacetime}$ is a randomized stopping time, i.e.\ \(\xi \in \RSTy{\initial}\), and for any non-negative measurable process $F: \Spacetime \to \R$ we have
  \begin{align}\label{RST2ST}
  \tsint F \d{\xi} = \E[(F \cdot \indicator{\Spacetime}) \circ \Phi] = \E[F(B,\tau) \cdot \indicator{\Time}(\tau)] \fullstop
  \end{align}
  For any $\xi \in \RSTy{\initial}$, we can find a $\timeindexed\G$-stopping time $\tau$ such that $\xi = \push{\Phi}{\P}$ and \eqref{RST2ST} holds.

  \(\xi\) is a finite randomized stopping time iff \(\tau\) is a.s.\ finite.
\end{BCHlemma}

\begin{proof}[Proof of Theorem \ref{thm:existence}]
We prove Version B of the theorem. Version A is a special case.
We show that Problem \probref{OptStop'} has a solution. To this end we show that the set \(\RSTyd{\initial}{\target}\) is compact (in the weak topology). From the fact that \(\cost\) is lower semicontinuous and bounded from below in an appropriate sense we then deduce by the Portmanteau theorem that the map 
\begin{align*}
\hat\cost & : \RSTyd{\initial}{\target} \rightarrow \lorc{-\infty,\infty} \\
\hat\cost(\zeta) & := \tsint \cost \d{\zeta}
\end{align*}
is lower semicontinuous and therefore that the infimum \( \inf_{\zeta \in \RSTyd{\initial}{\target}} \hat\cost(\zeta) \) is attained.

Now for the details.
On each of the spaces \(\Paths\) and \(\Time\) we are dealing with two topologies, one coming from the \autoref{def:RST} of randomized stopping times (to wit, the topology of uniform convergence on compacts on the space \(\Paths\) and the usual topology on \(\Time\)) and one coming from the assumptions in the statement of this theorem.
We can equip each of these spaces with the smallest topology which contains the two topologies in question.
These are again Polish topologies and they still generate the standard sigma-algebras on the respective spaces.
For the remainder of this proof all topological notions are to be understood relative to these topologies.
So the topology on \(\Spacetime\) is the product topology of these two topologies, and the weak topology on the space of measures on \(\Spacetime\) is to be understood relative to this product topology.
The cost function \(\cost\) of course remains lower semicontinuous and by \autoref{lem:cExp} the functions \((\omega,r) \mapsto F(r) \left( G(\omega) - \cExp{s}{G} \right)\) appearing in \autoref{def:RST} are continuous.

Note that for \(\xi \in \RSTyd{\initial}{\target}\) as \(\target\) has mass \(1\), so must \(\xi\) and \(\push{(\proj{\Paths})}{\xi}\), which together with \(\push{(\proj{\Paths})}{\xi} \leq \Blaw{0}{\initial}\) implies \(\push{(\proj{\Paths})}{\xi} = \Blaw{0}{\initial}\).
So we deduce
\newcommand{\CplWmu}{\Pi}%{\Cpl(\Blaw{0}{\initial},\target)}%
\begin{align*}
%\RSTyd{\initial}{\target} = \left\{ \xi \in \CplWmu : \tsint F(s) \left( G - \cExp{t}{G} \right )(\omega) \d{\xi}(\omega,s) = 0 \enspace \forall (t,F,G) \in \star \right\}
%cosmetic:
\RSTyd{\initial}{\target} = \left\{ \xi \in \CplWmu : \tsint F(s) \big( G - \E[G|\F^0_t] \big)(\omega) \d{\xi}(\omega,s) = 0 \enspace \forall (t,F,G) \in \star \right\}
\end{align*}
where
\begin{align*}
\xi \in \CplWmu & \iff  \push{(\proj{\Paths})}{\xi} = \Blaw{0}{\initial} \text{ and } \push{(\proj{\Time})}{\xi} = \target\\
(t,F,G) \in \star & \iff \parbox[t]{8.5cm}{ \(t > 0\), \(F \colon \Time \rightarrow \R\) is bounded and continuous in the usual topologies, and \(0\) outside \([0,t]\), \(G \colon \Paths \rightarrow \R\) is bounded and continuous as a function from the topology of uniform convergence on compacts.} %\fullstop
\end{align*}
\shortandlong{%
The set \( \CplWmu \) is compact by Prokhorov's Theorem and the fact that pushforwards are continuous maps between measure spaces.
}{%
The set \( \CplWmu \) is closed because pushforwards are continuous maps. We show that it is also tight, so that Prokhorov's Theorem implies that it is compact. Let \(\varepsilon > 0\) and choose compact sets \(K_1 \subseteq \Paths\), \(K_2 \subseteq \Time\) s.t.\ \(\Blaw{0}{\initial}(\compl{K_1}) < \frac \varepsilon 2 \) and \(\target(\compl{K_2}) < \frac \varepsilon 2\), then for all \(\xi \in \CplWmu\) we have \(\xi(\compl{(K_1 \times K_2)}) \leq \xi(\compl{K_1}\times\Time) + \xi(\Paths\times\compl{K_2}) < \varepsilon\).
}%
It remains to show that \(\RSTyd{\initial}{\target}\) is a nonempty closed subset.
It is nonempty because the product measure \(\Blaw{0}{\initial} \mtimes \target \in \RSTyd{\initial}{\target}\).
It is closed because, as noted, \((\omega,s) \mapsto F(s) \left( G - \cExp{t}{G} \right )(\omega)\) is continuous for all \((t,F,G) \in \star\).

Now we show that \( \hat\cost \) is lower semicontinuous.
The functions \( \cost^N := \cost \vee -N \) are each bounded from below and lower semicontinuous. By the Portmanteau theorem the maps \(\hat\cost^N := \zeta \mapsto \tsint \cost^N \d{\zeta}\) are lower semicontinuous.
On \(\RSTyd{\initial}{\target}\) they converge uniformly to \(\hat \cost\) because
\begin{align*}
\sup_\zeta \abs{\hat\cost(\zeta) - \hat\cost^N(\zeta)} \leq \sup_\zeta \tsint \abs{ \cost - \cost^N } \d{\zeta} \leq \sup_{\zeta \in \RSTyd{\initial}{\target}} \tsint \cost_- \cdot \indicator{\cost_- \geq N} \d{\zeta} \comma
\end{align*}
which converges to \(0\) as \(N\) goes to \(\infty\) by the uniform integrability assumption.
As a uniform limit of lower semicontinuous functions is again lower semicontinuous we see that \(\hat\cost\) is lower semicontinuous.
\end{proof}

\section{Geometry of the Optimizer}

This section is devoted to the proof of \autoref{thm:monotonicity}. The proof closely mimicks that of \autoref{BCH-GlobalLocal}/\autoref{BCH-GlobalLocal2} in \cite{BCH}. For the benefit of those readers already familiar with said paper we will first describe the changes required to the proofs there to make them work in our situation and then -- for the sake of a more self-contained presentation -- indulge in reiterating the main arguments and only citing results from \cite{BCH} that we can use verbatim.

\newcommand{\SGw}{\widehat\SG^\xi}
\begin{proof}[Sketch of differences in the proof of \autoref{thm:monotonicity} relative to {\cite[Theorem 5.7]{BCH}}%the proof of \autoref{BCH-GlobalLocal2} in \cite{BCH}
]
Again the strategy is to show that for a larger set \(\SGw \supseteq \SG\) we can find a set \(\Gamma \subseteq \Spacetime\) such that \(\SGw \cap \left(\init\Gamma \times \Gamma\right) = \emptyset\). The definition of \(\SGw\) must of course be adapted analoguously to the changes required to the definition of \(\SG\).

Apart from that the only real changes are to \cite[\autoref{BCH-Invisible2}]{BCH}. Whereas previously it was essential that the randomized stopping time \(\xi^{r(\omega,s)}\) is also a valid randomized stopping time of the Markov process in question when started at a different time but the same location \(\omega(s)\), we now need that \(\xi^{r(\omega,s)}\) will also be a randomized stopping time of our Markov process when started at the same time \(s\) but in a different place. Of course, when we are talking about Brownian motion both are true, but this difference is the reason why in the case of the Skorokhod embedding the right class of processes to generalize the argument to is that of Feller processes while in our setup we don't need our processes to be time-homogeneous but we do need them to be space-homogeneous.
That we are able to plant this \enquote{bush} \(\xi^{r(\omega,s)}\) in another location is what guarantees that the measure \(\xi_1^\pi\) defined in the proof of \autoref{BCH-Invisible2} of \cite{BCH} is again a randomized stopping time.

Whereas in the Skorokhod case the task is to show that the new better randomized stopping time \(\xi^\pi\) \emph{embeds} the same distribution as \(\xi\) we now have to show that the randomized stopping time we construct \emph{has} the same distribution as \(\xi\). The argument works along the same lines though -- instead of using that \(\left((\omega,s),(\eta,t)\right) \in \SGw\) implies \(\omega(s)=\eta(t)\) we now use that \(\left((\omega,s),(\eta,t)\right) \in \SGw\) implies \(s=t\).
\end{proof}

We now present the argument in more detail.

As may be clear by now, what we will show is that if \(\xi \in \RSTyd{\initial}{\target}\) is a solution of \probref{OptStop'}, then there is a \Sgood' set \(\Gamma \subseteq \Spacetime\) such that \(\SG \cap \left(\init{\Gamma} \times \Gamma\right) = \emptyset\). Using \autoref{lem:equivOpt} this implies \autoref{thm:monotonicity}.

\newcommand{\casesif}{& \quad \text{for} \enspace}%
We need to make some preparations.
To align the notation with \cite{BCH} and to make some technical steps easier it is useful to have another characterization of \Sgood' processes and sets. To this end define
\begin{definition}
\label{def:S}
\begin{align*}
S & := \bigcup_{t \in \Time} C([0,t]) \times \{t\} \\
r & : \Spacetime \rightarrow S \\
r(\omega,t) & := \left(\restr{\omega}{[0,t]},t\right)
\end{align*}
\(r\) has many right inverses. A simple one is
\newcommand{\rinv}{r'}%
\begin{align*}
\rinv & : S \rightarrow \Spacetime \\
\rinv (f,s) & := \left( t \mapsto
\begin{cases}
f(t) \casesif t \leq s \\
f(s) \casesif t > s
\end{cases} \enspace ,\enspace s \right) \fullstop
\end{align*}
We endow S with the sigma algebra generated by \(\rinv\).
\end{definition}

\cite[Theorem \ref{BCH-S2F}]{BCH}, which is a direct consequence of \cite[Theorem IV. 97]{DeMeA}, asserts that a process \(X\) is \Sgood' iff X factors as \(X=X'\circ r\) for a measurable function \(X' : S \rightarrow \R\). So a set \(D \subseteq \Spacetime\) is \Sgood' iff \(D = \invimage{r}{D'}\) for some measurable \(D' \subseteq S\).

Note that \(r(\omega,t) = r(\omega',t')\) implies \((\omega,t) \conc \theta = (\omega',t') \conc \theta\) and therefore
\begin{align*}
\SG & =\invimage{(r \ftimes r)}{\SG'}
\end{align*}
for a set \(\SG' \subseteq S \times S\) which is described by an expression almost identical to that in \autoref{def:SG}. Namely we can overload \(\conc\) to also be the name for the operation whose first operand is an element of \(S\), such that \((\omega,t) \conc \theta = r(\omega,t) \conc \theta\) and note that as \(\cost\) is measurable, \(\timeindexed\Fnull\)-adapted we can write \(\cost = \costS \circ r\) and thus get a cost function \( \costS \) which is defined on \(S\).

Given an optimal \(\xi \in \RSTyd{\initial}{\target}\) we may therefore rephrase our task as having to find a measurable set \(\Gamma \subseteq S\) such that \(\push{r}{\xi}\) is concentrated on \(\Gamma\) and that \(\SG' \cap \left(\init{\Gamma} \times \Gamma\right) = \emptyset\), where \(\init{\Gamma} := \left\{ (\restr{g}{[0,s]},s) : (g,t) \in \Gamma, s < t \right\}\).

Note that for \(\Gamma \subseteq S\) although \(\init{\left(\invimage{r}{\Gamma}\right)}\) is not equal to \( \invimage{r}{\init{\Gamma}} \) we still have \(\SG \cap \left(\invimage{r}{\init{\Gamma}} \times \invimage{r}{\Gamma}\right) = \emptyset \) iff \(\SG \cap \left(\init{(\invimage{r}{\Gamma})} \times \invimage{r}{\Gamma}\right) = \emptyset \).

One of the main ingredients of the proof of \cite[Theorem \ref{BCH-GlobalLocal}]{BCH} and of our \autoref{thm:monotonicity} is a procedure whereby we accumulate many infinitesimal changes to a given randomized stopping time \(\xi\) to build a new stopping time \(\xi^\pi\). The guiding intuition for the authors is to picture these changes as replacing certain \enquote{branches} of the stopping time \(\xi\) by different branches. Some of these branches will actually enter the statement of a somewhat stronger theorem (\autoref{thm:relative-monotonicity} below), so we begin by describing these.
Our way to get a handle on \enquote{branches} -- i.e.\ infinitesimal parts of a randomized stopping time -- is to describe them through a disintegration (wrt \(\Blaw{0}{\initial}\)) of the randomized stopping time.
We need the following statement from \cite{BCH} which should also serve to provide more intuition on the nature of randomized stopping times.

\begin{lemma}
\cite[Theorem \ref{BCH-thm:equiv RST}]{BCH}
\label{lem:equiv RST}
Let \(\xi\) be a measure on \(\Spacetime\). Then \( \xi \in \RSTy{\initial} \) iff there is a disintegration \((\xi_{\omega})_{\omega \in \Paths}\) of \(\xi\) wrt \(\Blaw{0}{\initial}\) such that \((\omega,t) \mapsto \xi_\omega([0,t])\) is measurable, \(\timeindexed\Fnull\)-adapted and maps into \([0,1]\).
\end{lemma}

Using \autoref{lem:equiv RST} above let us fix for the rest of this section both \(\xi \in \RSTyd{\initial}{\target}\) and a disintegration \(\left(\xi_{\omega}\right)_{\omega \in \Paths}\) with the properties above. Both \autoref{def:CRST} below and \autoref{thm:relative-monotonicity} implicitly depend on this particular disintegration and we emphasize that whenever we write \(\xi_{\omega}\) in the following we are always referring to the same fixed disintegration with the properties given in \autoref{lem:equiv RST}.
Note that the measurability properties of \(\left(\xi_{\omega}\right)_{\omega \in \Paths}\) imply that for any \( I \subseteq [0,s] \) we can determine \( \xi_\omega(I) \) from \( \restr{\omega}{[0,s]} \) alone. For \((f,s) \in S\) we will again overload notation and use \(\xi_{(f,s)} \) to refer to the measure on \([0,s]\) which is equal to \(\mrestr{\left(\xi_\omega\right)}{[0,s]}\) for any \( \omega \in \Paths \) such that \(r(\omega,s) = (f,s)\).

\begin{definition}[conditional randomized stopping time]
\label{def:CRST}
Let \((f,s) \in S\). We define a new randomized stopping time \(\xi^{(f,s)} \in \RSTt{s}\) by setting
\begin{align}
\label{eq:CRST}
\begin{aligned}
\xi^{(f,s)}_{\omega} & :=
\begin{cases}
\frac{1}{1-\xi_{(f,s)}([0,s])} \mrestr{\left(\xi_{(f,s) \conc \omega}\right)}{(s,\infty)} \casesif \xi_{(f,s)}([0,s]) < 1 \\
\delta_s \casesif \xi_{(f,s)}([0,s]) = 1
\end{cases} \\[2mm]
\tsint F \d\xi^{(f,s)} & := \tsiint F(\omega,t) \d\xi^{(f,s)}_\omega(t) \d\Blaw s 0 (\omega)
\end{aligned}
\end{align}
for all bounded measurable \(F : \Spacetime[s] \rightarrow \R\), i.e. \((\xi^{(f,s)}_{\omega})_{\omega \in \Paths[s]}\) is the disintegration of \( \xi^{(f,s)} \) wrt \(\Blaw{s}{0}\).

\end{definition}
Here \(\delta_s\) is the Dirac measure concentrated at \(s\). Really, the definition in the case where \( \xi_{(f,s)}([0,s]) = 1 \) is somewhat arbitrary -- it's more a convenience to avoid partially defined functions. What we will use is that \( \left(1-\xi_{(f,s)}([0,s])\right) \xi^{(f,s)}_\omega = \mrestr{\left(\xi_{(f,s) \conc \omega}\right)}{(s,\infty)} \).

\begin{definition}[relative Stop-Go pairs]
\label{def:SGxi}
The set \(\SG^\xi\) consists of all \(\left((f,t), (g,t)\right) \in S \times S\) (again the times have to match) such that either
\begin{align}
\label{eq:SGxi}
\costS(f,t) + \tsint \cost((g,t) \conc \theta, u) \d{\xi^{(f,t)}}(\theta,u) < \costS(g,t) + \tsint \cost((f,t) \conc \theta, u) \d{\xi^{(f,t)}}(\theta,u)
\end{align}
or any one of
\begin{enumerate}
\item\label{it:xilt1}    \(\xi^{(f,t)}\left(\Spacetime\right) < 1\) or \(\tsint s^{p_0} \d{\xi^{(f,t)}}(\theta,s) = \infty\)
\item\label{it:intinfty} the integral on the right hand side equals \(\infty\)
\item\label{it:intundef} either of the integrals is not defined
\end{enumerate}
holds.
We also define
\begin{align}
\label{eq:SGw}
\SGw := \SG^\xi \cup \left\{(f,s) \in S : \xi_{(f,s)}([0,s]) = 1 \right\} \times S
\end{align}
\end{definition}
\autoref{lem:xifs} below says that the numbered cases above are exceptional in an appropriate sense and one may consider them a technical detail. Note that when we say \(\left((f,t),(g,t)\right) \in \SG^\xi\) we are implicitly saying that \( \xi_{(f,t)}([0,t]) < 1 \).

Note that the sets \(\SG^\xi\) and \(\SGw\) are measurable (in contrast to \(\SG\), which may be more complicated).

\begin{definition}
We call a measurable set \(F \subseteq S\) evanescent if \(\invimage{r}{F}\) is evanescent, that is, if \(\Blaw{0}{\initial}\left(\proj{\Paths}\left[\invimage{r}{F}\right]\right) = 0\).
\end{definition}

\begin{lemma}
\cite[Lemma \ref{BCH-lem:xifs}]{BCH}
\label{lem:xifs}
Let \(F: \Spacetime \rightarrow \R\) be some measurable function for which \(\tsint F \d{\xi} \in \R\).
Then the following sets are evanescent.
\begin{itemize}
\item \(\left\{ (f,s) \in S : \xi^{(f,s)}\left(\Spacetime\right) < 1 \right\}\)
\item \(\left\{ (f,s) \in S : \tsint F((f,s) \conc \theta,u) \d{\xi^{(f,s)}}(\theta,u) \not\in \R \right\}\)
\end{itemize}
\end{lemma}
\begin{proof}See \cite{BCH}.\end{proof}

\begin{lemma}[{\cite[Lemma \ref{BCH-lem:SGsubSGW}]{BCH}}]
\label{lem:SGsubSGW}
\begin{align*}
\SG' \subseteq \SGw
\end{align*}
\end{lemma}
\begin{proof}
Can be found in \cite{BCH}. Note that they fix \(p_0 = 1\).
\end{proof}

\begin{theorem}
\label{thm:relative-monotonicity}
Assume that \( \xi \) is a solution of \probref{OptStop'}. Then there is a measurable set \( \Gamma \subseteq S \) such that \( \push{r}{\xi}(\Gamma) = 1 \) and
\begin{align}
\label{eq:relative-monotonicity}
\SGw \cap \left(\init{\Gamma} \times \Gamma\right) = \emptyset \fullstop
\end{align}
\end{theorem}

\newcommand{\JOIN}{\mathsf{JOIN}_{\initial}}
\newcommand{\Id}{\mathsf{Id}}

Our  argument follows \cite[Theorem \ref{BCH-GlobalLocal2}]{BCH}. We also need the following two auxilliary propositions, which in turn require some definitions.
\begin{definition}
\label{def:JOIN}
Let \(\upsilon\) be a probability measure on some measure space \(Y\). The set \(\JOIN(\upsilon)\) is the set of all subprobability measures \(\pi\) on \((\Spacetime) \times Y\) such that
\begin{align*}
& \push{(\proj{Y})}{\pi} \leq \upsilon \quad \text{ and} \\
& \push{(\proj{\Spacetime})}{\mrestr{\pi}{\Spacetime \times D}} \in \RSTy{\initial} \quad \text{for all measurable } D \subseteq Y \fullstop
\end{align*}
\end{definition}

\begin{proposition}
\label{Invisible2}
Let \( \xi \) be a solution of \probref{OptStop'}. Then \( \push{\left(r \ftimes \Id\right)}{\pi}(\SG^\xi) = 0 \) for all \( \pi \in \JOIN(\push{r}{\xi}) \).
\end{proposition}
Here we use \(\times\) to denote the Cartesian product map, i.e.\ for sets \(X_i,Y_i\) and functions \(F_i : X_i \rightarrow Y_i\) where \(i \in \{1,2\}\) the map \(F_1 \times F_2 : X_1 \times X_2 \rightarrow Y_1 \times Y_2\) is given by \((F_1 \times F_2)(x_1,x_2) = (F_1(x_1),F_2(x_2))\).
\autoref{Invisible2} is an analogue of \cite[Proposition \ref{BCH-Invisible2}]{BCH} and it is where the material changes compared to \cite{BCH} take place. We will give the proof at the end of this section.

\begin{proposition}
\cite[Proposition \ref{BCH-KeLe}]{BCH}
\label{KeLe}
Let \((Y, \upsilon)\) be a Polish probability space and let \( E \subseteq S \times Y \) be a measurable set. Then the following are equivalent
\begin{enumerate}
\item \( \push{\left(r \ftimes \Id\right)}{\pi}(E) = 0 \) for all \(\pi \in \JOIN(\upsilon)\)
\item \( E \subseteq (F \times Y) \cup (S \times N) \) for some evanescent set \(F \subseteq S\) and a measurable set \(N \subseteq Y\) which satisfies \(\upsilon(N) = 0\).
\end{enumerate}
\end{proposition}
\autoref{KeLe} is proved in \cite{BCH} and we will not repeat the proof here.

\begin{proof}[Proof of \autoref{thm:relative-monotonicity}]
Using \autoref{Invisible2} we see that \( \push{\left(r \ftimes \Id\right)}{\pi}(\SG^\xi) = 0 \) for all \( \pi \in \JOIN(\push{r}{\xi}) \). Plugging this into \autoref{KeLe} we find an evanescent set \(F_1 \subseteq S\) and a set \( N \subseteq S\) such that \(\push{r}{\xi}(N) = 0\) and \(\SG^\xi \subseteq (F_1 \times S) \cup (S \times N)\).
Defining for any Borel set \(E \subseteq S\) the analytic set
\begin{align*}
E^> := \left\{ (g,t) \in S : \exists s < t, \left(\restr{g}{[0,s]},s\right) \in E \right\}
\end{align*}
we observe that \( \init{\left(\compl{(E^>)}\right)} \subseteq \compl{E} \) and find \(\push{r}{\xi}(F_1^>) = 0\).

Setting \(F_2 := \left\{(f,s) \in S : \xi_{(f,s)}([0,s]) = 1 \right\}\) and arguing on the disintegration \(\left(\xi_\omega\right)_{\omega \in \Paths}\) we see that \( \push{r}{\xi}(F_2^>) = 0 \), so \(\push{r}{\xi}(F^>) = 0\) for \(F := F_1 \cup F_2\).

This shows that \(S \setminus (N \cup F^>)\) has full \(\push{r}{\xi}\)-measure. Let \(\Gamma\) be a Borel subset of that set which also has full \(\push{r}{\xi}\)-measure.

Then
\begin{align*}
\init{\Gamma} \times \Gamma & \subseteq \init{\left(\compl{(F^>)}\right)} \times \compl{N} \subseteq \compl{F} \times \compl{N} \text{ and}\\
\SGw & \subseteq (F \times S) \cup (S \times N)
\end{align*}
which shows \( \SGw \cap \left(\init{\Gamma} \times \Gamma\right) = \emptyset \).
\end{proof}

\begin{lemma}
\label{lem:helper}
If \( \alpha \in \RSTy{\initial} \) and \( G : \Spacetime \rightarrow [0,1] \) is \Sgood', then the measure defined by
\begin{align}
\label{eq:helper}
F \mapsto \tsint F(\omega, t) G(\omega, t) \d{\alpha}(\omega,t)
\end{align}
is still in \( \RSTy{\initial} \).
\end{lemma}
\begin{proof}
We use the criterion in \autoref{lem:equiv RST}.
Let \( (\alpha_\omega)_{\omega \in \Paths} \) be a disintegration of \( \alpha \) wrt \( \Blaw{0}{\initial} \) for which \( (\omega,t) \mapsto \alpha_\omega([0,t]) \) is \Sgood' and maps into \([0,1]\). Then \((\hat\alpha_\omega)_\omega\) defined by \( \hat\alpha_{\omega} := F \mapsto \tsint F(t) G(\omega,t) \d{\alpha_{\omega}}(t) \) is a disintegration of the measure in \eqref{eq:helper} for which \((\omega,t) \mapsto \hat\alpha_\omega([0,t]) \) is \Sgood' and maps into \([0,1]\).
\end{proof}

\begin{lemma}[Strong Markov property for RSTs]
\label{lem:markov}
Let \( \alpha \in \RSTy{\initial} \). Then
\begin{align*}
\tsint F(\omega,t) \d{\alpha}(\omega,t) = \tsiint F((\omega,t) \conc \tilde\omega, t) \d{\Blaw{t}{0}}(\tilde\omega) \d{\alpha}(\omega,t)
\end{align*}
for all bounded measurable \( F : \Spacetime \rightarrow \R \).
\end{lemma}
\begin{proof}
Using integral notation instead of the more conventional \(\E\), we may write the classical form of the strong markov property as
\begin{multline*}
\tsint G\left(\Theta_{\tau(\omega)}(\omega)\right) H(\omega) \cdot \indicator{\Time}(\tau(\omega)) \d\Blaw 0 \initial(\omega) = \\
\tsiint G(\tilde\omega) H(\omega) \cdot \indicator{\Time}(\tau(\omega)) \d \Blaw{\tau(\omega)}{\omega(\tau(\omega))}(\tilde\omega) \d \Blaw 0 \initial (\omega) 
\end{multline*}
for all bounded measurable \(G : \Paths \rightarrow \R\) and all bounded \(\Fnull_\tau\)-measurable \(H : \Paths \rightarrow \R\). Here \(\Theta_t\) is the function which cuts off the initial segment of a path up to time \(t\).
From this a simple monotone class argument shows that 
\begin{multline*}
\tsint K\left(\Theta_{\tau(\omega)}(\omega),\omega\right) \cdot \indicator{\Time}(\tau(\omega)) \d\Blaw 0 \initial(\omega) = \\
\tsiint K(\tilde\omega,\omega) \cdot \indicator{\Time}(\tau(\omega)) \d \Blaw{\tau(\omega)}{\omega(\tau(\omega))}(\tilde\omega) \d \Blaw 0 \initial (\omega) 
\end{multline*}
for all bounded \(\Fnull_\infty \atimes \Fnull_\tau\)-measurable \(K : \Paths \times \Paths\).

We may then choose for \(K(\tilde\omega, \omega)\) the function \(F(\eta, \tau(\omega))\) where the path \(\eta\) is created by cutting off the tail of \(\omega\) after time \(\tau(\omega)\) and attaching \(\tilde\omega\) in its place. Noting the relationship between \(\Blaw {\tau(\omega)} x\) and \(\Blaw {\tau(\omega)} 0\) we then get
\begin{multline*}
\tsint F(\omega,\tau(\omega)) \cdot \indicator{\Time}(\tau(\omega)) \d\Blaw 0 \initial(\omega) = \\
\tsiint F((\omega,\tau(\omega)) \conc \tilde\omega,\tau(\omega)) \cdot \indicator{\Time}(\tau(\omega)) \d \Blaw{\tau(\omega)}{0}(\tilde\omega) \d \Blaw 0 \initial (\omega) \fullstop
\end{multline*}

Using \autoref{lem:equivOpt} with \(\Omega = [0,1] \times \Paths\) and \(\G_t = \Borel{[0,1]} \atimes \F_t\) we find a \(\timeindexed\G\)-stopping time \(\tau\) s.t.\ we may write \(\alpha\) as \[ \alpha = \mrestr{\push{\big((y,\omega) \mapsto (\omega,\tau(y,\omega))\big)}{\Lebesgue \mtimes \Blaw 0 \initial}}{\Spacetime} \] (where \(\Lebesgue\) is Lebesgue measure on \([0,1]\)).
For a fixed \(y \in [0,1]\), \(\omega \mapsto \tau(y,\omega)\) is an \(\timeindexed\Fnull\)-stopping time, so we may apply the previous equation to these stopping times and integrate over \(y \in [0,1]\) to get
\begin{multline*}
\tsint F(\omega,\tau(y,\omega)) \cdot \indicator{\Time}(\tau(y,\omega)) \d(\Lebesgue \mtimes \Blaw 0 \initial)(y,\omega) = \\
\tsiint F((\omega,\tau(y,\omega)) \conc \tilde\omega,\tau(y,\omega)) \cdot \indicator{\Time}(\tau(y,\omega)) \d \Blaw{\tau(y,\omega)}{0}(\tilde\omega) \d (\Lebesgue \mtimes \Blaw 0 \initial) (y,\omega) \fullstop
\end{multline*}
Using the equation for \(\alpha\) we see that this is what we wanted to prove.
\end{proof}

\begin{lemma}[Gardener's Lemma]
\label{lem:gardener}
Assume that we have \(\xi \in \RSTyf{\initial}\), a measure \(\alpha\) on \(\Spacetime\) and two families \( \beta^{(\omega,t)} \), \( \gamma^{(\omega,t)} \), where \( (\omega,t) \in \Spacetime \), with \( \beta^{(\omega,t)}, \gamma^{(\omega,t)} \in \RSTtf{t} \) such that both maps
\begin{align*}
(\omega,t) & \mapsto \tsint \indicator{D}\left((\omega,t) \conc \tilde\omega,s\right) \d{\beta^{(\omega,t)}}(\tilde\omega,s) \enskip \text{ and } \\
(\omega,t) & \mapsto \tsint \indicator{D}\left((\omega,t) \conc \tilde\omega,s\right) \d{\gamma^{(\omega,t)}}(\tilde\omega,s)
\end{align*}
are measurable for all Borel \(D \subseteq \Spacetime\) and that 
\begin{align}
\label{eq:nonegativebushes}
\xi(D) - \tsiint \indicator{D}\left((\omega,t) \conc \tilde\omega,s\right) \d{\beta^{(\omega,t)}}(\tilde\omega,s) \d{\alpha}(\omega,t) \geq 0
\end{align}
for all Borel \(D \subseteq \Spacetime\).
Then for \(\hat\xi\) defined by
\begin{align*}
\tsint F \d{\hat\xi} := \tsint F \d{\xi} & - \tsiint F((\omega,t) \conc \tilde\omega,s) \d{\beta^{(\omega,t)}}(\tilde\omega,s) \d{\alpha}(\omega,t) \\
& + \tsiint F((\omega,t) \conc \tilde\omega,s) \d{\gamma^{(\omega,t)}}(\tilde\omega,s) \d{\alpha}(\omega,t)
\end{align*}
for all bounded measurable \(F\) we have \(\hat\xi \in \RSTyf{\initial}\).
\end{lemma}
\begin{remark}
The intuition behind the \nameref{lem:gardener} is that we are replacing certain branches \( \beta^{(\omega,t)} \) of the randomized stopping time \( \xi \) by other branches \( \gamma^{(\omega,t)} \) to obtain a new stopping time \( \hat\xi \). This process happens \emph{along} the measure \(\alpha\).
Note that \eqref{eq:nonegativebushes} implies that \(\tsint \indicator{D}\left((\omega,t) \conc \tilde\omega\right) \d{\Blaw{t}{0}}(\tilde\omega) \d{\alpha}(\omega,t) \leq \Blaw{0}{\initial}(D) \) for all Borel \(D \subseteq \Paths\).
The authors like to think of \(\alpha\) as a stopping time and of the maps \((\omega,t) \mapsto \beta^{(\omega,t)}\) and \((\omega,t) \mapsto \gamma^{(\omega,t)}\) as adapted (in some sense that would need to be made precise). As these assumptions aren't necessary for the proof of the \nameref{lem:gardener}, they were left out, but it might help the reader's intuition to keep them in mind.
\end{remark}
\begin{proof}[Proof of \autoref{lem:gardener}]
We need to check that the \(\hat\xi\) we define is indeed a measure, that \(\push{(\proj{\Paths})}{\hat\xi} = \Blaw{0}{\initial}\) and that \eqref{eq:RSTadapted} holds for \(\hat\xi\).

Checking that \(\hat\xi\) is a measure is routine -- we just note that \eqref{eq:nonegativebushes} guarantees that \(\hat\xi(D) \geq 0 \) for all Borel D.

Let \(G: \Paths \rightarrow \R \) be a bounded measurable function.
\begin{alignat*}{2}
\tsint G(\omega) \d{\hat\xi}(\omega,t) & = \tsint G(\omega) \d{\xi}(\omega,t) & & - \tsiint G((\omega,t) \conc \tilde\omega) \d{\beta^{(\omega,t)}}(\tilde\omega,s) \d{\alpha}(\omega,t) \\
& & & + \tsiint G((\omega,t) \conc \tilde\omega) \d{\gamma^{(\omega,t)}}(\tilde\omega,s) \d{\alpha}(\omega,t) \\
& = \tsint G \d{\Blaw{0}{\initial}} & & - \tsiint G((\omega,t) \conc \tilde\omega) \d{\Blaw{t}{0}} \d{\alpha}(\omega,t) \\
& & & + \tsiint G((\omega,t) \conc \tilde\omega) \d{\Blaw{t}{0}} \d{\alpha}(\omega,t) \\
& = \tsint G \d{\Blaw{0}{\initial}} &&
\end{alignat*}
Now let \(F : \Time \rightarrow \R\) and \(G: \Paths \rightarrow \R\) be bounded continuous functions, with \(F\) supported on \([0,r]\).
\begin{multline}
\label{eq:hatxiadapted}
\tsint F(t) \left(G - \cExp{r}{G}\right)(\omega) \d{\hat\xi}(\omega,t) = \tsint F(t) \left(G - \cExp{r}{G}\right)(\omega) \d{\xi}(\omega,t) \\
\shoveright{
- \tsiint F(s) \left(G - \cExp{r}{G}\right)((\omega,t) \conc \tilde\omega) \d{ \beta^{(\omega,t)}}(\tilde\omega,s) \d{\alpha}(\omega,t)} \\
- \tsiint F(s) \left(G - \cExp{r}{G}\right)((\omega,t) \conc \tilde\omega) \d{\gamma^{(\omega,t)}}(\tilde\omega,s) \d{\alpha}(\omega,t) 
\end{multline}
The first summand is \(0\) because \(\xi \in \RSTyf{\initial}\).
Looking at the second summand we expand the definition of \(\cExp{r}{G}\).
\begin{align*}
\cExp{r}{G}((\omega,t) \conc \tilde\omega)
& = \tsint G(((\omega,t) \conc \tilde\omega,r) \conc \theta) \d{\Blaw{r}{0}}(\theta) \\
& = \tsint G((\omega,t) \conc ((\tilde\omega,r) \conc \theta)) \d{\Blaw{r}{0}}(\theta)
\end{align*}
whenever \(t \leq r\), which is the case for those \(t\) which are relevant in the integrand above, because \( F(s) \neq 0 \) implies \( s \leq r \) and moreover \(\beta^{(\omega,t)}\) is concentrated on \((\tilde\omega,s)\) for which \( t \leq s \).

Setting \(\hat G^{(\omega,t)}(\tilde\omega) := G((\omega,t) \conc \tilde\omega)\) and \(\hat F^{(\omega,t)} := \restr{F}{\Time[t]}\) we can write
\begin{multline*}
\tsiint F(s) \left(G - \cExp{r}{G}\right)((\omega,t) \conc \tilde\omega) \d{ \beta^{(\omega,t)}}(\tilde\omega,s) \d{\alpha}(\omega,t) = \\
\tsint \indicator{[0,r]}(t) \tsint \hat F^{(\omega,t)}(s) \left( \smash{\hat G^{(\omega,t)}} - \cExp[t]{r}{\smash{\hat G^{(\omega,t)}}} \right)(\tilde\omega) \d{ \beta^{(\omega,t)}}(\tilde\omega,s) \d{\alpha}(\omega,t)
\end{multline*}
which is \(0\) because \(\beta^{(\omega,t)} \in \RSTtf{t}\) and therefore \[ \tsint \hat F^{(\omega,t)}(s) \left( \smash{\hat G^{(\omega,t)}} - \cExp[t]{r}{\smash{\hat G^{(\omega,t)}}} \right)(\tilde\omega) \d{ \beta^{(\omega,t)}}(\tilde\omega,s) = 0 \] for all \((\omega,t)\) and \(r \geq t\).
The same argument works for the third summand in \eqref{eq:hatxiadapted}.
\end{proof}

\newcommand{\x}[1]{\xi_{#1}^\pi}
\makeatletter
\newcommand{\pushright}[1]{\ifmeasuring@#1\else\omit\hfill$\displaystyle#1$\fi\ignorespaces}
\newcommand{\pushleft}[1]{\ifmeasuring@#1\else\omit$\displaystyle#1$\hfill\fi\ignorespaces}
\makeatother

\begin{proof}[Proof of \autoref{Invisible2}]
We prove the contrapositive. Assuming that there exists a \( \pi' \in \JOIN(\push{r}{\xi}) \) with \( \push{\left(r \ftimes \Id\right)}{\pi'}(\SG^\xi) > 0 \), we construct a \( \xi^\pi \in \RSTyd{\initial}{\target} \) such that \( \tsint \cost \d{\xi^\pi} < \tsint \cost \d{\xi} \).

If \(\pi' \in \JOIN(\push{r}{\xi})\), then for any two measurable sets \(D_1,D_2 \subseteq S\), because \(\mrestr{\pi'}{(\Spacetime) \times D_2} \in \RSTy{\initial}\) and by making use of \autoref{lem:helper} we can deduce that \(\push{(\proj{\Spacetime})}{\mrestr{\pi'}{\invimage{(r \ftimes \Id)}{D_1 \times D_2}}} \in \RSTy{\initial}\). Using the monotone classe theorem this extends to any measurable subset of \(S \times S\) in place of \(D_1 \times D_2\). So we can set \(\pi := \mrestr{\pi'}{\invimage{\left(r \ftimes \Id\right)}{\smash{\SG^\xi}}}\) and know that \(\push{(\proj{\Spacetime})}{\pi} \in \RSTy{\initial}\) and that \(\pi\) is concentrated on \(\SG^\xi\).

We will be using a disintegration of \( \pi \) wrt \(r(\xi)\), which we call \( \left(\pi_{(g,t)}\right)_{(g,t) \in S} \) and for which we assume that \(\pi_{(g,t)}\) is a subprobability measure for all \((g,t) \in S\). It will also be useful to assume that \( \pi_{(g,t)} \) is concentrated on the set \( \{ (\omega,s) \in \Spacetime : s = t \} \) not just for \( r(\xi) \)-almost all \( (g,t) \) but for all \( (g,t) \). Again this is no restriction of generality.
We will also push \( \pi \) onto \( \left(\Spacetime\right) \times \left(\Spacetime\right) \), defining a measure \(\bar\pi\) via
\begin{align*}
\tsint F \d{\bar\pi} := \tsiint F\left((\omega,s),((g,t) \conc \tilde\eta, t)\right) \d{\Blaw{t}{0}}(\tilde\eta) \d{\pi}\left((\omega,s),(g,t)\right)
\end{align*}
for all bounded measurable \(F\). Observe that by \autoref{lem:markov} the pushforward of \(\pi\) under projection onto the second coordinate (pair) is \(\xi\) and that a disintegration of \(\bar\pi\) wrt to \(\xi\) (again in the second coordinate) is given by \(\left(\pi_{r(\eta,t)}\right)_{(\eta,t) \in \Spacetime}\).
Let us name \(\push{(\proj{\Spacetime})}{\pi} =: \zeta \in \RSTy{\initial} \).
We will now use the \nameref{lem:gardener} to define two modifications \(\x0\), \(\x1\) of \(\xi\) such that \(\xi^\pi := \frac{1}{2}(\x0 + \x1)\) is our improved randomized stopping time.

\newcommand{\Surv}[2]{(1-\xi_{#1}([0,#2]))}

For all bounded measurable \(F : \Spacetime \rightarrow \R\) define
\begin{align*}
\tsint F \d{\x0} & := \tsint F \d{\xi} + \tsint \Surv{\omega}{s} \Big( - \tsint F((\omega,s) \conc \tilde\omega, u) \d{\xi^{r(\omega,s)}}(\tilde\omega,u) \\ & \pushright{ + F(\omega,s) \Big) \d{\zeta}(\omega,s)} \\
\tsint F \d{\x1} & := \tsint F \d{\xi} + \tsint \Surv{\omega}{s} \Big( -        F(\eta,t)                                                                 \\ & \pushright{ + \tsint F((\eta,t) \conc \tilde\omega,u) \d{\xi^{r(\omega,s)}}(\tilde\omega, u) \Big) \d{\bar\pi}((\omega,s),(\eta,t))} \fullstop
\end{align*}
The concatenation on the last line is well-defined \(\bar\pi\)-almost everywhere because \(\bar\pi\) is concentrated on \( \invimage{(r \ftimes r)}{\SG^\xi} \) and so in the integrand above \(s = t\) on a set of full measure.

We need to check that the \nameref{lem:gardener} applies in both cases. First of all observe that the product measure \( \Blaw{t}{0} \mtimes \delta_t \) is in \( \RSTtf{t} \) and that \autoref{lem:markov} implies
\begin{align*}
\tsint F(\omega,t) \d{\alpha}(\omega,t) = \tsiint F((\omega,t) \conc \tilde\omega, s) \d{\left(\Blaw{t}{0} \mtimes \delta_t\right)}(\tilde\omega,s) \d{\alpha}(\omega,t) \fullstop
\end{align*}
for any randomized stopping time \(\alpha\). So for \(\x0\) the measures \(\gamma^{(\omega,t)}\) are given by \( \Blaw{t}{0} \mtimes \delta_t \) and for \(\x1\) the measures \(\beta^{(\omega,t)}\) are given by \( \Blaw{t}{0} \mtimes \delta_t \).

For \( \x0 \) the measure along which we are replacing branches is given by
\begin{align*}
F \mapsto \tsint F(\omega,s) \Surv{\omega}{s} \d{\zeta}(\omega,s) \fullstop
\end{align*}
The branches \( \beta^{(\omega,s)} \) we remove are \(\xi^{r(\omega,s)} \).
We need to check that
\begin{align*}
\tsint F \d{\xi} - \tsint \Surv{\omega}{s} \tsint F((\omega,s) \conc \tilde\omega, u) \d{\xi^{r(\omega,s)}}(\tilde\omega,u) \d{\zeta}(\omega,s) \geq 0
\end{align*}
for all positive, bounded, measurable \( F : \Spacetime \rightarrow \R \).
Let us calculate.
\begin{multline*}
\tsint \Surv{\omega}{s} \tsint F((\omega,s) \conc \tilde\omega, u) \d{\xi^{r(\omega,s)}}(\tilde\omega,u) \d{\zeta}(\omega,s) = \\
%\tsint \Surv{\omega}{s} \tsiint F((\omega,s) \conc \tilde\omega, u) \d{\xi^{r(\omega,s)}_{\tilde\omega}}(u) \d{\Blaw{s}{0}}(\tilde\omega) \d{\zeta}(\omega,s) = \\
\tsiiint F((\omega,s) \conc \tilde\omega, u) \d{\left(\mrestr{(\xi_{(\omega,s) \conc \tilde\omega})}{(s,\infty)}\right)}(u) \d{\Blaw{s}{0}}(\tilde\omega) \d{\zeta}(\omega,s) = \\
\tsiint F(\omega, u) \d{\left(\mrestr{(\xi_{\omega})}{(s,\infty)}\right)}(u) \d{\zeta}(\omega,s) \leq 
\tsiint F(\omega, u) \d{(\xi_{\omega})}(u) \d{\zeta}(\omega,s) \leq \\ 
\tsiint F(\omega, u) \d{(\xi_{\omega})}(u) \d{\Blaw{0}{\initial}}(\omega) = 
\tsint F(\omega,u) \d{\xi}(\omega,u)
\end{multline*}
Here we first used the definition of \( \xi^{r(\omega,s)} \) and then \autoref{lem:markov} and finally that \(\push{(\proj{\Paths})}{\zeta} \leq \Blaw{0}{\initial}\).

For \(\x1\) we replace branches along
\begin{align*}
F & \mapsto \tsint F(\eta,t) \Surv{\omega}{s} \d{\bar\pi}\left((\omega,s),(\eta,t)\right) \\
  & = \tsint F(\eta,t) \tsint \Surv{\omega}{s} \d{\pi_{r(\eta,t)}}(\omega,s) \d{\xi}(\eta,t)
\fullstop
\end{align*}
The calculation above shows that
\begin{align*}
\tsint F \d{\xi} - \tsint \Surv{\omega}{s} F(\eta,t) \d{\bar\pi}\left((\omega,s),(\eta,t)\right) \geq 0
\end{align*}
for all positive, bounded, measurable \( F : \Spacetime \rightarrow \R \).
For \(\x1\) the branches \(\gamma^{(\eta,t)}\) that we add are given by
\begin{align*}
F \mapsto \frac{ \tsint \Surv{\omega}{s} \tsint F(\tilde\omega,u) \d{\xi^{r(\omega,s)}}(\tilde\omega,u) \d{\pi_{r(\eta,t)}}(\omega,s)
 }{ \tsint \Surv{\omega}{s} \d{\pi_{r(\eta,t)}}(\omega,s) }
\end{align*}
when \(\tsint \Surv{\omega}{s} \d{\pi_{r(\eta,t)}}(\omega,s) > 0\) and \( \delta_t \) otherwise (again, the latter is arbitrary).
In the more interesting case \( \gamma^{(\eta,t)} \) is an average over elements of \(\RSTtf{t}\) and therefore itself in \(\RSTtf{t}\). Here it is again crucial that for \( \pi_{r(\eta,t)} \)-almost all \((\omega,s)\) we have \(s = t\), otherwise we would be averaging randomized stopping times of our process started at unrelated times.

Putting this together we see that \(\xi^\pi := \frac{1}{2}(\x0 + \x1)\) is a randomized stopping time and that
\begin{multline}
\label{eq:Invisible2}
2 \tsint F \d{(\xi^\pi - \xi)} = 
\tsint \Surv{\omega}{s} \Big( F(\omega,s) - \tsint F((\omega,s) \conc \tilde\omega, u) \d{\xi^{r(\omega,s)}}(\tilde\omega,u) \\ - F(\eta,t) + \tsint F((\eta,t) \conc \tilde\omega,u) \d{\xi^{r(\omega,s)}}(\tilde\omega, u) \Big) \d{\bar\pi}((\omega,s),(\eta,t))
\end{multline}
for all bounded measurable  \( F : \Spacetime \rightarrow \R \).
Specializing to \(F(\omega,s) = G(s)\) for \(G : \Time \rightarrow \R\) bounded measurable we find that
\begin{align*}
\tsint G(s) \d{(\xi-\xi^\pi)}(\omega,s) = 0 \text{ ,}
\end{align*}
again because for \(\bar\pi\)-almost all \(\left((\omega,s),(\eta,t)\right)\) we have \(s=t\). This shows that \(\xi^\pi \in \RSTyd{\initial}{\target}\).

\makeatletter
\newcommand{\specialcell}[1]{\ifmeasuring@#1\else\omit$\displaystyle#1$\ignorespaces\fi}
\makeatother

We now want to extend \eqref{eq:Invisible2} to \(\cost\). We first show that \eqref{eq:Invisible2} also holds for \(F: \Spacetime \rightarrow \R\) which are measurable and positive and for which \(\tsint F \d{\xi} < \infty\).
To see this, approximate such an \(F\) from below by bounded measurable functions (for which \eqref{eq:Invisible2} holds) and note that by previous calculations both 
\begin{align*}
\tsint \! \Surv{\omega}{s} \tsint F((\omega,s) \conc \tilde\omega, u) \d{\xi^{r(\omega,s)}}(\tilde\omega,u) \d{\bar\pi}((\omega,s),(\eta,t)) \leq \tsint F \d{\xi} & < \infty \\
\specialcell{\text{and } \hfill \tsint \Surv{\omega}{s} F(\eta,t) \d{\bar\pi}((\omega,s),(\eta,t)) \leq \tsint F \d{\xi}} & < \infty \fullstop
\end{align*}
Looking at positive and negative parts of \(\cost\) and using Assumption \assref{ass:wellposed} to see that \( \tsint c_{-} \d{(\xi^\pi-\xi)} \in \R \) we get that indeed \eqref{eq:Invisible2} holds for \(F = c\).

Now we will argue that the integrand in the right hand side of \eqref{eq:Invisible2} is negative \(\bar\pi\)-almost everywhere. This will conclude the proof.

By inserting an \(r\) in appropriate places we can read off from \autoref{def:SGxi} what it means that \(\bar\pi\) is concentrated on \(\invimage{(r \ftimes r)}{\SG^\xi}\). In the course of verifying that \eqref{eq:Invisible2} applies to \(\cost\) we already saw that cases \ref{it:intinfty} and \ref{it:intundef} in \autoref{def:SGxi} can only occur on a set of \(\bar\pi\)-measure \(0\). \autoref{lem:xifs} excludes case \ref{it:xilt1} \(\bar\pi\)-almost everywhere. This means that \eqref{eq:SGxi} holds \(\bar\pi\)-almost everywhere -- or more correctly, that for \(\bar\pi\)-a.a. \(((\omega,s),(\eta,t))\) we have \(s=t\) and
\begin{multline}
\label{eq:integrand_negative}
\cost(\omega,s) - \tsint \cost((\omega,s) \conc \tilde\omega, u) \d{\xi^{r(\omega,s)}}(\tilde\omega,u) \\
- \cost(\eta,t) + \tsint \cost((\eta,t) \conc \tilde\omega, u) \d{\xi^{r(\omega,s)}}(\tilde\omega,u) < 0 \comma
\end{multline}
completing the proof.
\end{proof}

\section{Variations on the Theme}

We proceed to prove \autoref{cor:maxbarrier}. This is closely modelled on the treatment of the Azema-Yor embedding in \cite[Theorem \ref{BCH-thm:AY}]{BCH}.
As is the case there we run into a technical obstacle, though one which can be overcome by combining the ideas we have already seen in slightly new ways.

To demonstrate the problem let us begin an attempt to prove \autoref{cor:maxbarrier}. Again, we read off \(\cost(\omega,t) = -\rmax \omega t\), with \(\rmax\omega t = \sup_{s \leq t}\omega(s)\). We may use \autoref{thm:existence} to find a solution \(\tau\) of the problem \probref{OptStopmax} and we use \autoref{thm:monotonicity} to find a set \(\Gamma \subseteq \Spacetime\) for which \(\P[(B,\tau) \in \Gamma] = 1\) and \(\SG \cap (\init{\Gamma}\times\Gamma) = \emptyset\).
Now we would like to apply \autoref{lem:geometry} with \(Y_t(\omega) = \omega(t) - \rmax\omega t\), as proposed by \autoref{cor:maxbarrier}, so we want to prove that \(\omega(t) - \rmax \omega t < \eta(t) - \rmax \eta t\) implies \(((\omega,t),(\eta,t)) \in \SG\).
Let us do the calculations. We start with an \(\timeindexed[s][t]{\F^t}\)-stopping time \(\sigma\), for which \(\Blaw{t}{0}(\sigma = t) < 1\), \(\Blaw{t}{0}(\sigma = \infty) = 0\) and for which both sides in \eqref{eq:SG} are defined and finite. To reduce clutter, let us name \(\push{(\omega \mapsto (\omega,\sigma(\omega)))}{\Blaw{t}{0}} =: \alpha\), so that \eqref{eq:SG}, which we want to prove, reads
\begin{align}
\label{eq:SG1}
- \rmax \omega t + \tsint \rmax{((\omega,t) \conc \theta)}{s} \d{\alpha}(\theta,s) < 
- \rmax \eta   t + \tsint \rmax{((\eta  ,t) \conc \theta)}{s} \d{\alpha}(\theta,s)
\end{align}
We may rewrite the left hand side as
\begin{multline*}
\tsint \Big(\rmax \omega t \vee \big(\omega(t) + \rmax \theta s\big)\Big) - \rmax \omega t \d{\alpha}(\theta,s) = \\
\tsint 0 \vee \big(\omega(t) - \rmax \omega t + \rmax \theta s \big) \d{\alpha}(\theta,s) \fullstop
\end{multline*}
For the right hand side we get the same expression with \(\omega\) replaced by \(\eta\). Looking at the integrands we see that if
\begin{align}
\label{eq:AY_simple_case}
0 < \eta(t) - \rmax \eta t + \rmax \theta s
\end{align}
then
\begin{align*}
0 \vee \big( \omega(t) - \rmax \omega t + \rmax \theta s \big) < 
0 \vee \big( \eta(t) - \rmax \eta t + \rmax \theta s \big) \comma
\end{align*}
but in the other case
\begin{align*}
0 \vee \big( \omega(t) - \rmax \omega t + \rmax \theta s \big) = 0 =
0 \vee \big( \eta(t) - \rmax \eta t + \rmax \theta s \big) \fullstop
\end{align*}
So if \eqref{eq:AY_simple_case} holds for \((\theta,s)\) from a set of positive \(\alpha\)-measure, then we proved what we wanted to prove. But if \( \rmax \theta s \leq \rmax \eta t - \eta(t) \) for \(\alpha\)-a.a.\ \((\theta,s)\) then in \eqref{eq:SG} we have equality instead of strict inequality.

As in \cite[Theorem \ref{BCH-thm:AY}]{BCH}, one way of getting around this is to introduce a secondary optimization criterion.
One way to explain the idea of secondary optimization is to think about what happens if, instead of considering a cost function \(\cost : \Spacetime \rightarrow \R\) we consider a cost function \(\cost : \Spacetime \rightarrow \R^n\).
Of course, to be able to talk about optimization, we will then want to have an order on \(\R^n\). For reasons that should become clear soon, we decide on the lexicographical order.
For the case \(n = 2\) that we are actually interested in for \autoref{cor:maxbarrier} this means that 
\begin{align*}
(x_1,x_2) \leq (y_1,y_2) \iff x_1 < y_1 \text{ or } (x_1 = y_1 \text{ and } x_2 \leq y_2) \fullstop
\end{align*}

We claim that \autoref{thm:monotonicity} is still true if we replace \(\cost : \Spacetime \rightarrow \R\) by \(\cost : \Spacetime \rightarrow \R^n\) and read any symbol \(\leq\) which appears between vectors in \(\R^n\) as the lexicographic order on \(\R^n\) (and of course likewise for all the derived symbols and notions \(<\), \(\geq\), \(>\), \(\inf\), etc.).
Moreover, the arguments are exactly the same.
Indeed the crucial part that may deserve some mention is at the end of the proof of \autoref{Invisible2}, where we use the assumption that \eqref{eq:integrand_negative} holds on a set of positive measure, i.e.\ that the integrand is \( < 0 \) on a set of positive measure, and that the integrand is \(0\) outside that set, to conclude that the integral itself must be \( < 0 \). This implication is also true for the lexicographical order on \(\R^n\).
One more detail to be aware of is that integrating functions which map into \(\R^2\) may give results of the form \((\infty,x)\), \((x,-\infty)\), etc.
In the case of a one-dimensional cost function we excluded such problems by making Assumption \assref{ass:wellposed}.
What we really want in the proof of \autoref{Invisible2} is that \(\tsint \cost \d{\xi}\) and \(\tsint \cost \d{\xi^\pi} \) should be finite. Clearly a sufficient condition to guarantee this is to replace Assumption \assref{ass:wellposed} by
\begin{enumerate}
\item[(\ref{ass:wellposed}')] \( \E[\cost(B,\tau)] \in \R^n \) for all stopping times \( \tau \sim \target \).
\end{enumerate}
This is not the most general version possible but it will suffice for our purposes.

To get an existence result we may assume that \(\cost=(\cost_1,\cost_2)\) is component-wise lower semicontinuous and that both \(\cost_1\) and \(\cost_2\) are bounded below (in either of the ways described in the two versions of \autoref{thm:existence}).
Note that -- because we are talking about the lexicographic order -- \(\xi \in \RSTyd{\initial}{\target}\) is a solution of \probref{OptStop'} for \(\cost\) iff \(\xi\) is a solution of \probref{OptStop'} for \(\cost_1\) and among all such solutions \(\xi'\), \(\xi\) minimizes \(\tsint \cost_2 \d{\xi'}\).
By \autoref{thm:existence} in the form that we have already proved the set of solutions of \probref{OptStop'} for \(\cost_1\) is non-empty. It is also a closed subset of a compact set and therefore itself compact. This allows us to reiterate the argument that we used in the proof of \autoref{thm:existence} to find inside this set a minimizer of \(\xi' \mapsto \tsint \cost_2 \d{\xi'}\). This minimizer is the solution of \probref{OptStop'} for \(\cost\).

With this in hand we may pick up our
\begin{proof}[Proof of \autoref{cor:maxbarrier}]
The same arguments as in the proof of \autoref{cor:Bbarrier} apply, so we may assume that our probability space satisfies Assumption \assref{ass:randomization}.
We start with a cost function \(\cost(\omega,t) := (\cost_1(\omega,t),\cost_2(\omega,t)) := (-\rmax \omega t, (\rmax \omega t - \omega(t))^3)\).
\(\Lnorm{3}{\cost_1(B,\tau)} \leq \Lnorm{3}{\rmaxP{\abs{B}} \tau} \leq K_1 \Lnorm{3/2}{\tau}^{1/2}\), by the Burkholder-Davis-Gundy inequalities, so \((\cost_1)_-\) satisfies the uniform integrability condition and \(\E[\cost(B,\tau)]\) is finite for all stopping times \(\tau \sim \target\).
\(c_2 \geq 0\) and by the Burkholder-Davis-Gundy inequalities \(\E[c_2(B,\tau)] \leq \E[(\rmax B \tau)^3] \leq K_1 \E[\tau^{3/2}] = K_1 \tsint t^{3/2} \d{\target}(t)\) for some constant \(K_1\). The last term is finite by assumption.

By our discussion in the preceding paragraphs we find a solution \(\tau\) of \probref{OptStop} for \(\cost\) and a \Sgood' set \(\Gamma \subseteq \Spacetime\), for which \(\P[(B,\tau) \in \Gamma] = 1\) and \(\SG \cap (\init{\Gamma}\times\Gamma) = \emptyset\), where now \( ((\omega,t),(\eta,t)) \in \SG\) iff for all \(\timeindexed[s][t]{\F^t}\)-stopping times \(\sigma\) for which \(\Blaw{t}{0}(\sigma = t) < 1\), \(\Blaw{t}{0}(\sigma = \infty) = 0\), \(\tsint \sigma^{3/2} \d{\Blaw{t}{0}} < \infty\), setting \(\alpha := \push{(\omega \mapsto (\omega,\sigma(\omega)))}{\Blaw{t}{0}}\) we have that
either equation \eqref{eq:SG1} holds or
\begin{align}
\label{eq:SG1eq}
- \rmax \omega t + \tsint \rmax{((\omega,t) \conc \theta)}{s} \d{\alpha}(\theta,s) = 
- \rmax \eta   t + \tsint \rmax{((\eta  ,t) \conc \theta)}{s} \d{\alpha}(\theta,s)
\end{align}
\hspace*{\fill} and \hspace*{\fill}
\begin{align}
\label{eq:SG2}
\cost_2(\omega,t) - \tsint \cost_2((\omega,t) \conc \theta, s) \d{\alpha}(\theta,s) < 
\cost_2(\eta  ,t) - \tsint \cost_2((\eta  ,t) \conc \theta, s) \d{\alpha}(\theta,s) 
\fullstop
\end{align}
Now we want to apply \autoref{lem:geometry}, so we want to show that \(\omega(t) - \rmax \omega t < \eta(t) - \rmax \eta t\) implies \(((\omega,t),(\eta,t)) \in \SG\).
We already dealt with the case where \(\alpha\) is such that \eqref{eq:AY_simple_case} holds on a set of positive \(\alpha\)-measure. We now deal with the other case, so we have
\begin{align}
\label{eq:AY_less_simple_case}
\rmax \theta s \leq \rmax \eta t - \eta(t) < \rmax \omega t - \omega(t)
\end{align}
for \(\alpha\)-a.a.\ \((\theta,s)\) and we know that \eqref{eq:SG1eq} holds.
We show that \eqref{eq:SG2} holds.
Because of \eqref{eq:AY_less_simple_case}, \( \rmax{((\omega,t) \conc \theta)}{s} = \rmax \omega t \), and so \( \cost_2((\omega,t) \conc \theta, s) = (\rmax \omega t - \omega(t) - \theta(s))^3 \). We calculate the left hand side of \eqref{eq:SG2}.
\begin{multline*}
\tsint (\rmax \omega t - \omega(t))^3 - (\rmax \omega t - \omega(t) - \theta(s))^3 \d{\alpha}(\theta,s) = \\
\tsint 3 (\rmax \omega t - \omega(t))^2 \theta(s) - 3 (\rmax \omega t - \omega(t)) (\theta(s))^2 + (\theta(s))^3 \d{\alpha}(\theta,s) = \\
(\omega(t) - \rmax \omega t) 3 \tsint (\theta(s))^2 \d{\alpha}(\theta,s) + \tsint (\theta(s))^3 \d{\alpha}(\theta,s)
\end{multline*}
Here the Burkholder-Davis-Gundy inequalities show that both \(\tsint (\theta(s))^3 \d{\alpha}(\theta,s)\) and \( \tsint (\theta(s))^2 \d{\alpha}(\theta,s) \) are finite so that we may split the integral and they also show that \(\{\tilde B_{\sigma \wedge T}: T \geq t\}\) is uniformly integrable so that by the optional stopping theorem \(\tsint \theta(s) \d{\alpha}(\theta,s) = 0\). (\(\tilde B\) is again Brownian motion started in \(0\) at time \(t\) on \(\Paths[t]\).)

For the right hand side of \eqref{eq:SG2} we get the same expression with \(\omega\) replaced by \(\eta\).
This concludes the proof that \(\omega(t) - \rmax \omega t < \eta(t) - \rmax \eta t\) implies \(((\omega,t),(\eta,t)) \in \SG\) and \autoref{lem:geometry} gives us barriers \(\Rcl\), \(\Rop\) such that for their hitting times \(\taucl\), \(\tauop\) by \(B_t - \rmaxP B t\) we have \(\taucl \leq \tau \leq \tauop\) a.s.

Again we want to show that \(\taucl = \tauop\) a.s.\ and that they are actually stopping times.
Again we do so by showing that they are both a.s.\ equal to the hitting time of the closure of the respective barrier.
If \(\closure\Rcl \cap (\{0\} \times \Time) = \emptyset\) then this works in exactly the same way as in \autoref{lem:lessnasty}.
(This time we define \(\closure\tau_\varepsilon := \inf \{ t > 0 : (B^\varepsilon_t(\omega) - \rmaxP{(B^\varepsilon)}{t}(\omega), t) \in \closure\Barrier \}\) where \(B^\varepsilon_t(\omega) := B_t(\omega) + A(t) \varepsilon\).)
If \(\closure\Rcl \cap (\{0\} \times \Time) \neq \emptyset\) then \((B^\varepsilon_t(\omega) - \rmaxP{(B^\varepsilon)}{t}(\omega), t) \in \closure\Barrier\) and \(t>0\) need not imply \( B_t(\omega) - \rmaxP B t (\omega) < B^\varepsilon_t(\omega) - \rmaxP{(B^\varepsilon)}{t}(\omega)\), which is essential for the topological argument showing that the hitting time of \(\Barrier\) is less than or equal \(\closure\tau_\varepsilon\).
But if \(\closure\Rop \cap (\{0\} \times \Time) = \closure\Rcl \cap (\{0\} \times \Time) \neq \emptyset\), then \(\taucl\) and \(\tauop\) are both almost surely \(\leq T\) where \(T := \inf \{ t > 0 : (0,t) \in \closure\Rop\}\), so in the step where we show that the hitting time of \(\Barrier\) is less than \(\closure\tau_\varepsilon\) we can argue under the assumption that \(\closure\tau_\varepsilon(\omega) < T\). In this case we do have that \((B^\varepsilon_t(\omega) - \rmaxP{(B^\varepsilon)}{t}(\omega), t) \in \closure\Barrier\) and \(t>0\) implies \( B_t(\omega) - \rmaxP B t (\omega) < B^\varepsilon_t(\omega) - \rmaxP{(B^\varepsilon)}{t}(\omega)\).
\end{proof}

\begin{remark}\label{rem:otherpictures}
We hope that the proofs of \autoref{cor:Bbarrier} and \autoref{cor:maxbarrier} have given the reader some idea of how to apply the main results of this paper to arrive at barrier-type solutions of constrained optimal stopping problems, as depicted in \autoref{fig:Test}.

We would like to conclude by giving a couple of pointers to the interested reader who may want to work through the proofs corresponding to the remaining pictures in \autoref{fig:Test}.

For the problem of minimizing \(\E[\rmaxP B \tau]\), it may actually happen that the times \(\taucl, \tauop\) from \autoref{lem:geometry} do not coincide. Specifically one has to expect this to happen on a non-negligible set when \(\Rcl\) contains parts of the time axis which \(\Rop\) does not contain. Under these circumstances an optimizer may turn out to be a true randomized stopping time, with a proportion of a path hitting the time axis at a certain point needing to be stopped while the rest continues. In this situation the picture alone does not completely describe the optimal stopping time.

For the problems involving absolute values one needs to make a minor modification in the proof of \autoref{Invisible2}. Specifically one can allow \enquote{mirroring} the paths which are \enquote{transplanted} using the \nameref{lem:gardener}. This leads to a slightly different definition of Stop-Go pairs, which is perhaps most easily described by saying that in \autoref{fig:StopGo} the green paths which are stoppen by \(\sigma\) may be flipped upside-down on either side.
\end{remark}

\bibliographystyle{alpha}
\bibliography{optstop}{}

\end{document}